\documentclass[journal,twoside,web]{ieeecolor}
\usepackage{generic}
\usepackage{cite}

\usepackage{amsmath, amssymb, amsthm}
\usepackage{graphicx}
\usepackage{comment}
\usepackage{placeins}
\usepackage{enumerate}
\usepackage{hyperref}

\usepackage[short]{optidef}	
\newtheorem{theorem}{Theorem}

\newtheorem{proposition}{Proposition}
\newtheorem{definition}{Definition}

\newtheorem{corollary}{Corollary}

\newtheorem{example}{Example}

\usepackage{xcolor}

\newcommand{\R}{\mathbb{R}}
\newcommand{\A}{\mathcal{A}}

\newcommand{\dd}{\mathrm{d}}

\newcommand{\I}{{\mathcal{I}}}

\newcommand{\B}{{\mathcal{B}}}

\newcommand{\BNLP}{{\mathrm{BNLP}}}

\definecolor{darkgreen}{rgb}{0.0,0.5,0.0}

\def\BibTeX{{\rm B\kern-.05em{\sc i\kern-.025em b}\kern-.08em
    T\kern-.1667em\lower.7ex\hbox{E}\kern-.125emX}}
\begin{document}
\title{Real-Time Algorithms for Model Predictive Control of Hybrid Dynamical Systems
}
\author{
	Armin Nurkanovi\'c,
	Anton Pozharskiy,
	Moritz Diehl
\thanks{
   Armin Nurkanovi\'c and Anton Pozharskiy are with the Department of Microsystems Engineering (IMTEK), University of Freiburg, Germany, Moritz Diehl is with the Department of Microsystems Engineering (IMTEK) and Department of Mathematics, University of Freiburg, Germany. Corresponding author: \texttt{armin.nurkanovic@imtek.uni-freiburg.de}. \\ 
	This research was supported by DFG via projects 504452366 (SPP 2364), 560056112 (robust MPC), 535860958 (ALeSCo) and 525018088 (MAWERO), and by BMWK via 03EN3054B.}}

\maketitle

\begin{abstract}
	Model predictive control (MPC) of hybrid dynamical systems is challenging because the associated optimization problem is nonsmooth and the resulting feedback law is discontinuous. 
	This paper develops real-time MPC algorithms for nonlinear hybrid systems modeled as dynamical complementarity systems. 
	The resulting optimal control problems are formulated as mathematical programs with complementarity constraints (MPCCs).
	We show that the solution map of parametric MPCCs is discontinuous, and that standard nonlinear-programming-based approaches may become infeasible when the hybrid system switches. 
	To address this, we introduce three real-time hybrid MPC schemes whose feedback phase solves a quadratic program with complementarity constraints per sample, yielding local discontinuous piecewise affine approximations of the MPC feedback law. 
	Moreover, we derive continuity and differentiability results for parametric MPCCs, and establish conditions under which the approximation error of our new hybrid MPC algorithms remains uniformly bounded despite solution discontinuities. 
	The algorithms are demonstrated on a robotic manipulation example, where contact sequences are discovered online.
\end{abstract}

\begin{IEEEkeywords}
Predictive control, hybrid system, optimal control, optimization.
\end{IEEEkeywords}

\vspace{-0.34cm}

\section{Introduction}\label{sec:intro}
\IEEEPARstart{M}{odel} predictive control (MPC) is a feedback-control strategy that directly handles constraints and multivariable dynamics by repeatedly solving finite-horizon optimal control problems in real-time.
For linear and mildly nonlinear systems with continuous controls, it is a mature and widely applied methodology, with successful applications ranging from agile drone racing~\cite{Romero2022} to large-scale chemical plants~\cite{Qin2003}.

However, the dynamics of many practical systems cannot be described by smooth nonlinear differential equations because they exhibit nonsmooth \textit{state-dependent events}, such as contact making and breaking in robotics~\cite{Brogliato2020,Posa2014}, phase transitions in chemical processes, and flow reversals in transport networks~\cite{Baumrucker2009}.
This leads to \textit{hybrid dynamical systems}, characterized by an interplay of discrete and continuous dynamics, which makes the application of MPC computationally challenging.
The theoretical properties of hybrid MPC are well studied~\cite{Lazar2009,Sanfelice2026}, and several computational advances exist~\cite{Bemporad1999b,Nurkanovic2023f,Kong2024}.
However, unlike in the smooth MPC case~\cite{Diehl2009c}, there is currently no established generic real-time MPC framework for hybrid systems that provides accurate solution approximations at limited computational cost.
In this paper, we propose several hybrid MPC algorithms aimed at minimizing the computational load per sampling instant, while keeping approximation errors small and uniformly bounded.
We focus on the dynamic complementarity systems (DCS) formalism for hybrid systems, which covers most applications of interest and has favorable computational properties; cf. Section~\ref{sec:problem_formulation}.
We consider state-dependent discontinuities only, and do not treat discrete control decisions, which require a different set of modeling and analytic tools~\cite{Sager2005,Kirches2010y}.

\begin{figure}[t]
	\vspace{0.2cm}
	\centering
	\includegraphics[width=0.47\textwidth]{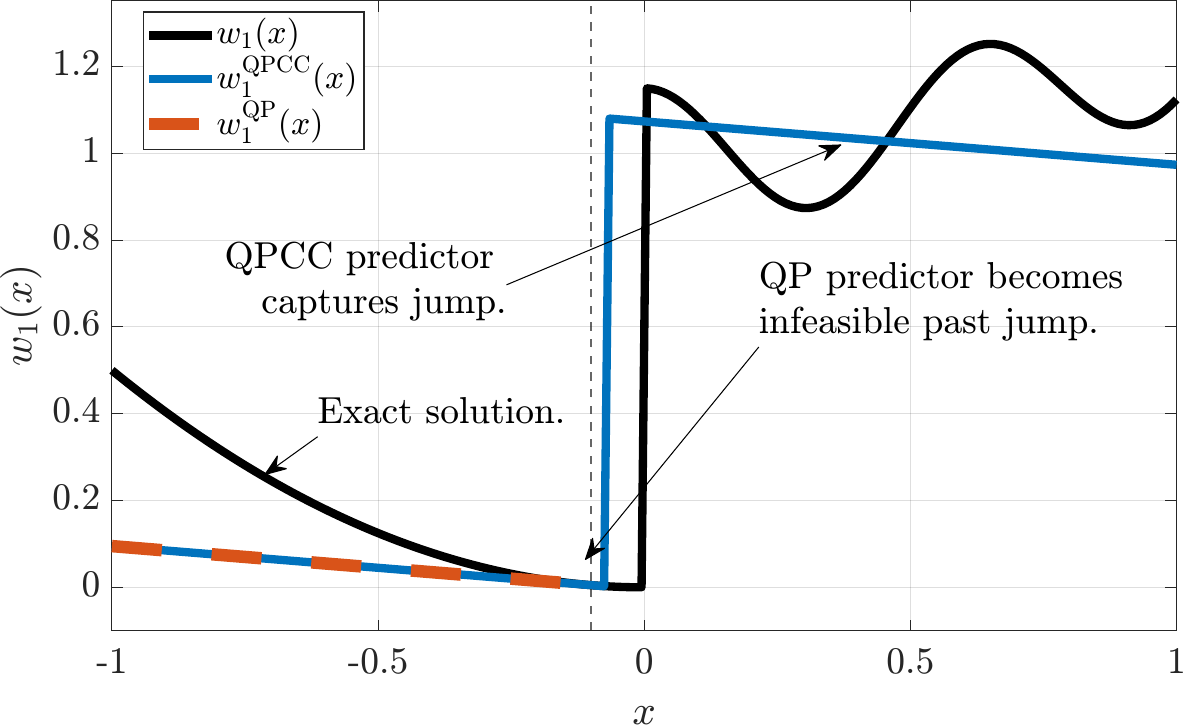}
	\vspace{-0.2cm}
	\caption{Solution map of a parametric MPCC~\eqref{eq:patthfollow_mpcc}, and its QPCC and QP approximations at $\bar x=-0.1$. 
	The QP predictor matches the QPCC as long as the complementarity active sets agree, and becomes infeasible afterwards. 
	The QPCC predictor captures the solution jump.
	}
\label{fig:control_law}
	\vspace{-0.6cm}
\end{figure}

Classical nonlinear MPC solves a sequence of related smooth parametric nonlinear programs (NLPs), where the parameter is typically the latest state estimate $x_k$.
Established real-time MPC algorithms exploit the similarity of neighboring problems and therefore restrict the amount of online computation~\cite{Diehl2009c}.
Most such methods combine the sequential quadratic programming (SQP)~\cite{Diehl2001} method with the implicit function theorem (IFT)~\cite{Fiacco1983} applied to the necessary optimality conditions of the parametric NLP~\cite{Jittorntrum1984,Ralph1995} to obtain a (piecewise) affine approximation of the optimal solution map.
This reduces the online feedback delay to the time needed for solving a single linear system or quadratic program (QP)~\cite{Diehl2001,Zavala2009} after the new state estimate $x_k$ becomes available; cf. Sec.~\ref{sec:mpc}.

Discrete-time optimal control problems subject to DCS lead to parametric mathematical programs with complementarity constraints (MPCCs)~\cite{Luo1996,Scheel2000,Kim2020}.
As for parametric NLPs, the SQP method can be applied to NLP reformulations of MPCCs, provided that a QP solver, capable of detecting a linearly independent active set basis, is available~\cite{Fletcher2006}.
However, interior-point QP solvers, which often perform better in MPC applications~\cite{Kouzoupis2018}, generally do not satisfy this requirement.
Even if such a QP solver is available, classical SQP-based methods face additional difficulties.
In particular, standard constraint qualifications (cf. Sec.~\ref{sec:nlp_theory}), which are necessary for the IFT-based sensitivity results underlying classical real-time MPC algorithms~\cite{Jittorntrum1984,Ralph1995}, are violated by MPCCs at all feasible points.
On the other hand, compared to parametric NLPs, the parametric properties of MPCCs have received considerably less attention~\cite{Scheel2000,Hu2002,Guo2014}.
We show that the solution map of parametric MPCCs, and thus the resulting hybrid MPC feedback law, is generically discontinuous and cannot be well approximated by QP steps alone.
We show that the QP subproblem is typically infeasible when the hybrid system must switch for a new state estimate $x_{k}$, cf. Figure~\ref{fig:control_law}.

The central premise of this paper is that real‑time hybrid MPC must solve at least quadratic programs with complementarity constraints (QPCCs), and not ordinary QPs.
This view is consistent with several recent works that employ QPCCs in related settings~\cite{Kim2025,Li2025,Aydinoglu2024}.
Section~\ref{sec:path_follow} illustrates that a QPCC can produce discontinuous solution approximations and, unlike a QP, can detect new switching sequences.
Furthermore, solving a sequence of QPCCs~\cite{Scholtes2004,Nurkanovic2026}, i.e., the SQPCC method, is a natural extension of the SQP method~\cite{Nocedal2006} to MPCCs.
Despite their difficulty, MPCCs and QPCCs can now be solved reliably enough to support tailored real-time hybrid MPC algorithms~\cite{Ralph2004,Nurkanovic2024b,Kanzow2015,Pozharskiy2026} (cf. Section~\ref{sec:examples}).
Furthermore, MPCC-tailored methods based on continuous optimization techniques are usually orders of magnitude faster than integer methods applied to MPCC reformulations~\cite{Hempel2017,Nurkanovic2025,Nurkanovic2022a,Hall2024}.

\paragraph{Contributions.}
This paper contributes as follows:
\begin{itemize}
	\item We extend existing sensitivity results for parametric MPCCs~\cite{Scheel2000,Hu2002}.
	In particular, under weaker assumptions than in~\cite{Scheel2000,Hu2002}, we characterize when solution selections remain piecewise continuously differentiable under complementarity active-set changes, when they are locally unique (Theorem~\ref{th:mpcc_sesitivity}), and when discontinuities necessarily occur (Theorem~\ref{th:mpcc_jump}).
	
	\item In Proposition~\ref{prop:directional_derivaies}, we show how to compute directional derivatives of MPCC solution maps by solving a single QPCC, thereby contributing to the field of differentiable optimization for MPCCs.
	
	\item In Section~\ref{sec:path_follow}, we analyze the local approximation properties of the SQPCC method applied to parametric MPCCs with continuously varying parameters, and derive, in Theorem~\ref{th:path_follow_sqpcc_formal}, conditions under which the approximation error remains uniformly bounded.
	
	\item Section~\ref{sec:mpc} introduces three distinct real-time hybrid MPC algorithms. 
	The first approach does a single SQPCC iteration per sampling instant, generalizing classical real-time iterations~\cite{Diehl2001}.
	The second follows the advanced-step idea of Zavala and Biegler~\cite{Zavala2009}, where the next state ${x}^{\mathrm{pred}}_{k+1}$ is predicted and an advanced problem is solved before the true state ${x}_{k+1}$ becomes available, and the feedback step then solves a QPCC to capture possible discontinuities.
	Further variants solve the advanced problem only approximately, in the spirit of advanced-step real-time iterations introduced in~\cite{Nurkanovic2019a}.
	Based on Theorem~\ref{th:path_follow_sqpcc_formal}, we discuss error bounds for the new algorithms.
	None of the proposed algorithms requires an initial switching-sequence guess or heuristics for updating it, switches are discovered automatically online.
	
	\item We provide an open-source prototype implementation of the proposed algorithms in the software package \texttt{nosnoc}~\cite{Nurkanovic2022b}, and illustrate in Sec.~\ref{sec:examples} the proposed real-time hybrid MPC algorithms on a robotics example. 
\end{itemize}

\paragraph{Outline.}
Sec.~\ref{sec:problem_formulation} describes the problem classes considered in this paper.
Sec.~\ref{sec:nlp_theory} gives basic background on NLPs and MPCCs, and Sec.~\ref{sec:sqpcc} recalls a recent local convergence result for the SQPCC method.
Sec.~\ref{sec:mpcc_sensitivity} studies parametric MPCCs and differentiability properties of their solution maps.
In Section~\ref{sec:path_follow}, we analyze SQPCC-based path-following algorithms and derive uniform error bounds.
Based on these results, Sec.\ref{sec:mpc} introduces several real-time hybrid MPC algorithms and discusses their error bounds.
Finally, Sec.~\ref{sec:examples} illustrates the approach on a numerical example, and Sec.~\ref{sec:conclusion} discusses directions for future research.

\paragraph{Notation.} 
For the Jacobian of a function $f: \R^n \to \R^m$ we will use the notation $\nabla f(w) = \frac{\partial f}{\partial w}(w)^\top \in \R^{n \times m}$.
An open ball centered at $\bar w \in \R^n$ with radius $\varepsilon >0$ is denoted by $\B_{\varepsilon}(\bar w) := \{ w \in \R^n \mid \| w- \bar w \| < \varepsilon \}$.
The concatenation of two vectors $ x \in \R^{n_x}, y \in \R^{n_y}$ is denoted by $ (x,y) := [x^\top, y^\top]^\top$. 
The concatenation of several vectors is defined accordingly.
We use subscripts $x_k$ to indicate the time index. 
In static problems the subscripts in $x_j$ denote the $j$th component, and if we have both, we use $x_{k,j}$.
Superscripts in $w^k$ are used for the iteration index.
Further, for function evaluations at the $k$th iterate we use shorthands of the form of $H^k := H(w^k)$.
We use the symbol $L$ for the usual NLP Lagrangian, and calligraphic $\mathcal{L}$ for the MPCC--Lagrangian.


\section{Problem formulation}\label{sec:problem_formulation}
MPC requires the repeated solution of finite-horizon discrete-time optimal control problems (OCPs).
In this paper, we consider a hybrid OCP subject to a discrete-time dynamic complementarity system (DCS), to be solved at every sampling time $t_k$:
\vspace{-0.05cm}
\begin{mini!}[2]
	{\substack{s,y,u}}{\textstyle \sum_{i=0}^{N-1} \ell(s_i,u_i) + E(s_N)}
	{\label{eq:discrete_time_ocp}}{}
	\addConstraint{s_0}{= x_k}
	\addConstraint{s_{i+1}}{= \phi_f(s_i,y_i,u_i), \label{eq:discrete_time_ocp_dcs_dyn}}{\quad i = 0,\ldots,N-1}
	\addConstraint{0}{= \phi_{\mathrm{int}}(s_i,y_i,u_i),}{\quad i = 0,\ldots,N-1}
	\addConstraint{0 \leq \phi_G(y_i) \perp \phi_H(y_i)}{\geq 0, \label{eq:discrete_time_ocp_dcs_cmp}}{\quad i = 0,\ldots,N-1}
	\addConstraint{g(s_i,y_i, u_i)}{\leq 0, }{\quad i = 0,\ldots,N-1 \label{eq:discrete_time_ocp_path}}
	\addConstraint{r(s_N)}{{\leq 0} \label{eq:discrete_time_ocp_term}.}
\end{mini!}
The vector $x_k \in \R^{n_x}$ corresponds to the discrete-time differential states at time $t_k$, and $u_i \in \R^{n_u}$ the discrete-time control input on the interval $t\in [t_i,t_{i+1})$.
The vectors $s_i \in \R^{n_x}$ collect the state trajectory for $i = 0,\ldots,N$, and the vectors $y_i$ collect all algebraic and possibly internal variables of the numerical integration method.
All these variables are collected in the vectors
$s = (s_0,\ldots,s_N)$, 
$y = (y_0,\ldots,y_{N-1})$, and 
$u = (u_0,\ldots,u_{N-1})$. 
The functions $\phi_f$, $\phi_{\mathrm{int}}$, $\phi_G$, and $\phi_H$ define the discrete-time dynamic complementarity system (DCS).
The term $\ell:\R^{n_x} \times \R^{n_u} \to \R$ is the running cost function for the $i$th stage, and $E:\R^{n_x} \to \R$ is the terminal cost function.
The functions $g: \R^{n_x} \times \R^{n_u} \to \R^{m_g}$ and $r: \R^{n_x} \to \R^{m_r}$ define the path and terminal constraints in~\eqref{eq:discrete_time_ocp_path} and~\eqref{eq:discrete_time_ocp_term}, respectively.

The DCS in~\eqref{eq:discrete_time_ocp_dcs_dyn}--\eqref{eq:discrete_time_ocp_dcs_cmp} covers numerous hybrid dynamical systems, including discretizations of Filippov systems~\cite{Filippov1964}, piecewise smooth systems~\cite{Hempel2017,Nurkanovic2024a,Nurkanovic2024c,Brogliato2020}, complementarity Lagrangian systems (for robotics with friction and impacts)~\cite{Brogliato2016,Posa2014}, projected dynamical systems and (time-varying) sweeping processes~\cite{Brogliato2006,Pozharskiy2025}, relay systems~\cite{Johansson2002}, mixed logical dynamical systems and max–min–plus scaling systems~\cite{Heemels2001}, and hybrid systems with hysteresis~\cite{Nurkanovic2022a}.
See~\cite[Section~3]{Brogliato2020}, \cite{Nurkanovic2023f}, and \cite{Heemels2001} for equivalence results between DCS and the various other hybrid systems formalisms.

In our exposition, we highlighted DCS-induced MPCCs~\eqref{eq:discrete_time_ocp}.
However, the proposed algorithms and our theoretical results cover also other sources of complementarity constraints, e.g., general piecewise smooth functions~\cite[Chapter~11]{Biegler2010} or logical path, vanishing~\cite{Achtziger2008} and state-triggered constraints~\cite{Szmuk2020}.

Here we assume that a discrete-time OCP~\eqref{eq:discrete_time_ocp} is given, but highlight that the time discretization of hybrid systems requires tailored transcription methods~\cite{Baumrucker2009,Nurkanovic2024a} to avoid the generation of spurious local minima~\cite{Stewart2010,Nurkanovic2020}.

By defining $w = [s^\top, y^\top, u^\top]^\top \in \R^n$ and suitable functions 
$f:\R^n \to \R$ for the objective, 
$g:\R^n \to \R^{m_g}$ for the equality constraints, 
a selection matrix $M \in \R^{m_h} \times \R^{n_x}$,
$h:\R^n \to \R^{m_h}$ for the inequality constraints, and 
$G,H:\R^n \to \R^m$ for the complementarity functions, 
the discrete-time OCP~\eqref{eq:discrete_time_ocp} can be written compactly as a generic parametric MPCC:
\begin{mini!}[2]
	{\substack{w\in \R^{n}}}{f(w)\label{eq:mpcc_ocp_obj}}
	{\label{eq:mpcc_ocp}}{}
	\addConstraint{h(w)+M x_k}{=0 \label{eq:mpcc_ocp_eq}}
	\addConstraint{g(w)}{\leq 0 \label{eq:mpcc_ocp_ineq}}
	\addConstraint{0 \leq G(w) \perp H(w)}{\geq 0. \label{eq:mpcc_ocp_comp}}
\end{mini!}
In ideal hybrid MPC, once a new state estimate $x_{k}$ becomes available, one would instantaneously compute the global optimal solution $\bar w(x_k)$ and apply the first-stage control $\bar u_0(x_k)$ to the plant.
Since this is impossible in practice, as established in nonlinear MPC~\cite{Diehl2009c}, we derive algorithms that compute approximations $u_0(x_k) \approx \bar u_0(x_k)$ of local minimizers of the MPCC~\eqref{eq:mpcc_ocp} with as little computation effort as possible, while ensuring that the numerical error remains uniformly bounded, even if the hybrid dynamical system switches.

Note that if the complementarity active sets of the DCS in~\eqref{eq:discrete_time_ocp} were known in advance for some $x_k$, then the OCP would reduce to a standard NLP rather than an MPCC.  
In other words, the switching sequence of the hybrid system would be fixed. 
In Section~\ref{sec:mpcc_theory}, we will see that such NLPs are called branch NLPs of the MPCC.
However, for every new state estimate $x_{k+1}$, such a fixed sequence does not have to remain feasible. 
Hence, the goal of hybrid MPC is to implicitly discover a locally optimal switching sequence online.
\section{Preliminaries on nonlinear programming}\label{sec:nlp_theory}
To state our novel results in a self-contained way, we review basic optimality and sensitivity results for standard nonlinear programs (NLP), as well as basic MPCC-tailored notions.
\subsection{Regularity and local optimality}
Consider the following parametric NLP:
\begin{mini!}[2]
	{\substack{w \in \R^{n}}}{f(w,p)\label{eq:nlp_obj}}
	{\label{eq:nlp}}{}
	\addConstraint{h(w,p)}{=0 \label{eq:nlp_eq}}
	\addConstraint{g(w,p)}{\leq0, \label{eq:nlp_ineq}}
\end{mini!}
where $f: \R^n \times \R^{n_p} \to \R$, and $h:\R^n \times \R^{n_p} \to \R^{m_h},\ g:\R^{n} \times \R^{n_p} \to \R^{m_g}$ are twice continuously differentiable. 
The vector $p \in \R^{n_p}$ is a fixed parameter, and to keep the notation lighter we drop the dependencies on $p$ whenever clear from the context.
The Lagrange multipliers for the equality and inequality constraints are denoted by $\lambda \in \R^{m_h}$ and $\mu \in \R^{m_g}$, respectively.
The feasible set of the NLP~\eqref{eq:nlp} is denoted by $\Omega := \{ w\in \R^n \mid h(w) = 0, g(w) \leq 0\}$.
The Lagrange function $L: \R^n \times \R^{m_h} \times \R^{m_g} \to \R$ of the NLP~\eqref{eq:nlp} reads as:
\begin{align*}
	L(w,\lambda, \mu) &= f(w) + \lambda^\top h(w) + \mu^\top g(w). 
\end{align*}

The Karush–Kuhn–Tucker (KKT) conditions read as:
\begin{subequations}\label{eq:kkt}
	\begin{align}
		&\nabla f(w) + \nabla h(w) \lambda + \nabla g(w) \mu = 0, \\
		&h(w) = 0,\\
		&0 \leq -g(w) \perp \mu \geq 0.
	\end{align}
\end{subequations}
A point $\bar z:= (\bar w,\bar \lambda,\bar \mu)$ satisfying these conditions is called a KKT point, and $\bar w$ is called a stationary point.
If in addition a constraint qualification holds (defined below), the KKT conditions are first-order necessary conditions for optimality.
The active set at a feasible point $w \in \Omega$ is defined as
	$\A(w) = \{ i \in \{1,\ldots,m_g \} \mid g_i(w) = 0\}$.
Its complement $\A^{\mathrm{C}}(w) = \{1,\ldots, m_g \} \setminus \A(w)$ is the set of inactive constraints.
The strictly and weakly active sets at a KKT point $(x,\lambda,\mu)$ are defined as
	$\A_+(w,\mu) = \{ i \in \A(w) \mid \mu_i > 0\},\,
	\A_0(w,\mu) = \{ i \in \A(w) \mid \mu_i = 0\}$.
We say that strict complementarity holds for a KKT point if 
\(
\mu_i> 0, \ \forall i \in \A(w)
\).
Next, we define some basic regularity concepts in nonlinear programming.

\begin{definition}[LICQ]
	At feasible point $w \in \Omega$, the linear independence constraint qualification (LICQ) is said to hold if the set of vectors $\{ \nabla h_i(w) \mid i = 1,\ldots, m_h\} \cup \{ \nabla g_i(w) \mid i \in \A(w)\}$ are linearly independent.
\end{definition}

The LICQ is sufficient for unique Lagrange multipliers.
We use the notation $\nabla g_{\I}(w) \in \R^{n \times |\I|}$ for a matrix whose columns are the vectors $\nabla g_i(w),\ i \in \I \subseteq \{1,\ldots,m_g\}$.
Next, we define some second-order sufficient conditions (SOSC).
The tangent space of all equality and strictly active constraints is called the strong critical cone:
\begin{align*}
	\mathcal{C}^\mathrm{S}(w,\mu) = \{ d\in \R^n \mid \nabla h(w)^\top d \!=\! 0,\, \nabla g_{\A_+(w,\mu)}(w)^\top d \!=\! 0 
	\}.
\end{align*}
The usual critical cone is defined by adding restrictions w.r.t. weakly active constraints:
$\mathcal{C}(w,\mu) = \{d \in \mathcal{C}^{\mathrm{S}}(w,\mu)  \mid \nabla g_{\A_0(w,\mu)}(w)^\top d \leq 0 \}$.

If strict complementarity holds we have that $\mathcal{C}^{\mathrm{S}}(w,\mu) = \mathcal{C}(w,\mu)$.
We will frequently use the so-called strong second-order sufficient conditions (SSOSC).
\begin{definition}
	The SSOSC is said to hold at a stationary point $w \in \Omega$ if there exist corresponding multipliers $(\lambda,\mu)$ such that $(w,\lambda,\mu)$ is a KKT point, and if the following inequality holds:
	\begin{align}\label{eq:ssosc}
		d^\top \nabla_{ww}^2 L(w,\lambda,\mu) d > 0,\ \forall d \in \mathcal{C}^{\mathrm{S}}(w,\mu) \setminus\{0\}.
	\end{align}
\end{definition}
A less restrictive notion is the usual SOSC, which is obtained if we replace $\mathcal{C}^{\mathrm{S}}(w,\mu)$ by $\mathcal{C}(w,\mu)$ in \eqref{eq:ssosc}.

\subsection{Stability of parametric NLP solutions}\label{sec:nlp_sesnitivity}
Next, we explicitly focus on the dependence on the parameter $p \in \R^{n_p}$ in the NLP~\eqref{eq:nlp}. 
In Sec.~\ref{sec:mpcc_sensitivity}, we extend these results to MPCCs. 
This is crucial for deriving the error bounds of hybrid MPC algorithms.
We restate a classical result on the stability of solutions of parametric NLPs~\cite{Fiacco1983,Jittorntrum1984}.
\begin{theorem}[Theorem~2.4.5, \cite{Fiacco1983}]\label{th:fiacco_nlp_sensitivity}
	If $f$, $g$, and $h$ are twice continuously differentiable in $(w,p)$ in a neighborhood of $(\bar w,\bar p)$, if the SSOSC hold for problem \eqref{eq:nlp} at a KKT point $(\bar w,\bar{\lambda},\bar{\mu})$ corresponding to a fixed $\bar p$, and if the LICQ holds at $\bar w$, then:
	\begin{enumerate}[(a)]
		
		\item $\bar w$ is an isolated local optimal solution of~\eqref{eq:nlp} for $p = \bar p$, with $(\bar \lambda,\bar \mu)$ being the unique Lagrange multiplier vector associated with $\bar{p}$;
		
		\item there exist a neighborhood of $\bar p$, and in this neighborhood there exists a unique continuous vector function
		\(
		z(p) = (w(p), \lambda(p),  \mu(p)) 
		\)
		satisfying the SSOSC for a local minimum of	problem \eqref{eq:nlp} such that
		$z(\bar p) = (\bar w ,\bar \lambda ,\bar \mu )$, and hence $w(p)$ is a locally unique minimizer of \eqref{eq:nlp} with the parameter $p$, and with associated unique Lagrange multipliers $\lambda(p)$ and $\mu(p)$;

		\item the solution map $z(p)$ is Lipschitz in $p$, i.e., there exist constants $0<\gamma<\infty$ and $\delta>0$ such that for every $p \in \B_{\delta}(\bar p)$ it holds that:
		\begin{align}\label{eq:lipschitz_solution_fiacco}
			\|z(p)-\bar z\|\leq \gamma \|p-\bar{p}\|
		\end{align}
		
		\item the optimal value function $V(p)=f(w(p)),p)$ is differentiable with respect
		to $p$ at $p=\bar{p}$, with
		\(
		\nabla_p V(\bar{p})	= \nabla_p L(\bar w,\bar{\lambda},\bar{\mu},\bar{p});
		\)
		
		\item in any direction $v\neq 0$, the uniquely determined, one-sided directional derivative $D_v z(p)$ exists at $\bar{p}$.
	\end{enumerate}
\end{theorem}
In fact, the function $z(p)$ is a piecewise continuously differentiable ($\mathrm{PC}^1$) near $\bar p$, i.e., it is continuous and there exists a finite family of $\mathrm{C}^1$ functions $z^1(p),\ldots,z^N(p)$, defined on a neighborhood of $\bar p$, such that $z(p)\in\{z^1(p),\ldots,z^N(p)\}$ for all $p$ in that neighborhood~\cite{Ralph1995}.

It follows from the continuity of the solution map $z(p)$ of~\eqref{eq:nlp}, that strictly active constraints remain active and inactive constraints remain inactive for all $p \in \B_{\rho}(\bar p)$ for some $\rho > 0$.
The size of this neighborhood depends on the associated multipliers and the slack of the inactive constraints at $\bar p$.
We state this elementary fact explicitly for later reference. 
\begin{corollary}[Corollary 2.1, \cite{Nurkanovic2026}]\label{lem:active_set_stabilization}
	Regard the parametric NLP~\eqref{eq:nlp} and suppose the assumptions of Theorem~\ref{th:fiacco_nlp_sensitivity} hold true for~\eqref{eq:nlp}, with $(\bar w, \bar \lambda, \bar \mu)$ being the solution at a fixed $\bar p$.
	Then there exist a constant $\rho>0$, such that for every $p \in \B_{\rho} (\bar p)$ it holds that
	\(
	\mathcal{A}_+(\bar w,\bar\mu) \!= \!\mathcal{A}_+(w(p),\mu(p)),
	\,
	\mathcal{A}^\mathrm{C}(\bar w)\! =\! \mathcal{A}^\mathrm{C}(w(p)).
	\)
\end{corollary}

Theorem~\ref{th:fiacco_nlp_sensitivity}(e) states the existence of directional derivatives in all directions $v \in \R^{n_p}$.
In general, these directional derivatives are computed by solving a certain convex quadratic program (QP), cf.~\cite[Theorem~3]{Jittorntrum1984} or \cite[Theorem~1]{Ralph1995}.
If the parameter in~\eqref{eq:nlp} appears only linearly in the equality constraints, e.g., $g(w,p) = \tilde g(w) + Mp = 0$, 
the directional derivative $D_v z(\bar p)$ can be recovered from the QP solution used in the SQP method~\cite[Theorem~3.6]{Diehl2001}.
This insight is the basis for real-time nonlinear MPC algorithms such as the real-time iteration (RTI)~\cite{Diehl2001} and many of its variations~\cite{Bock2007,Nurkanovic2019a,Zanelli2021c,TranDinh2012b}, which compute a SQP iteration per sampling time.
Sec.~\ref{sec:mpc} generalizes such ideas to the hybrid MPC case.

\subsection{MPCC theory}\label{sec:mpcc_theory}
The MPCC~\eqref{eq:mpcc_ocp} can be reformulated into a standard NLP. 
For a more compact notation, we drop in this section the parameter dependency via the term $Mx_{k}$ in~\eqref{eq:mpcc_ocp_eq}.
By replacing \eqref{eq:mpcc_ocp_comp} by a set of inequality constraints in~\eqref{eq:mpcc_nlp_compG}--\eqref{eq:mpcc_nlp_bilinear}, we obtain the NLP:
\allowdisplaybreaks
\begin{mini!}[2]
	{\substack{w\in \R^{n}}}{f(w)\label{eq:mpcc_nlp_obj}}
	{\label{eq:mpcc_nlp}}{}
	\addConstraint{h(w)}{=0 \label{eq:mpcc_nlp_eq}}
	\addConstraint{g(w)}{\leq0 \label{eq:mpcc_nlp_ineq}}
	\addConstraint{G(w)}{\geq0 \label{eq:mpcc_nlp_compG}}
	\addConstraint{H(w)}{\geq0 \label{eq:mpcc_nlp_compH}}
	\addConstraint{G_i(w)H_i(w)}{\leq0, \quad i=1,\ldots,m. \label{eq:mpcc_nlp_bilinear}}
\end{mini!}
This and other NLP reformulations of the MPCC~\cite{Scheel2000} result in highly degenerate NLPs.
In particular, the constraints \eqref{eq:mpcc_nlp_compG}-\eqref{eq:mpcc_nlp_bilinear} lead to violation of the LICQ (and also weaker conditions) at all feasible points~\cite{Scheel2000}.
As a result, classical NLP theory and algorithms usually cannot be applied successfully to this problem class.
Instead, applying tailored MPCC notions and algorithms resolves most of the issues.
They all rely on the piecewise nature of the MPCC's feasible set.

\subsubsection{First-order optimality conditions}
Algebraic conditions for stationarity of regular NLPs are given by the KKT conditions~\eqref{eq:kkt}. 
For MPCCs, there are more nuances, and there are several first-order stationarity concepts, which we review briefly here. 
For surveys see~\cite{Scheel2000,Kim2020}.

We call an NLP regular if it satisfies some first- and second-order regularity conditions, e.g., LICQ and SSOSC.
Most of the MPCC theory relies on defining auxiliary regular NLPs related to the MPCC.
For this purpose, we define the following index sets which depend on a feasible point $w$ of the MPCC~\eqref{eq:mpcc_ocp}:
\begin{align*}
	\I_{0+}(w) &= \{i \in \{1, \ldots, m \} \mid G_i(w)=0, H_i(w)>0\},\\
	\I_{+0}(w) &=	\{i \in \{1, \ldots, m\} \mid G_i(w)>0, H_i(w)=0\},\\
	\I_{00}(w) &= \{i \in \{1, \ldots, m \} \mid G_i(w)=0, H_i(w)=0\}.
\end{align*}
If clear from the context, we omit the argument $w$ in the index sets.
We call indices $i \in \I_{00}(w)$ degenerate index pairs, and indices $i \in \{1,\ldots,m\} \setminus \I_{00}(w)$ nondegenerate.
Most theoretical difficulties arise in the presence of degenerate indices.
To compactly define branch NLPs below, we define a few more index sets.
Let $\tilde{\mathcal{P}}(w) = \{ \tilde \I(w)  \subseteq  \I_{00}(w)\}$ denote the powerset of the degenerate indices.
Then, at a feasible point $w\in \Omega$, we define the set of index sets: 
\begin{align*}
	\mathcal{P}(w)  = \{ \I \subseteq \{1,\ldots,m\}\mid \exists\, \tilde{\I}  \in \tilde{\mathcal{P}}(w),\, \I = \I_{0+}(w)\cup \tilde{\I}\}.
\end{align*}
Moreover, for a given $\I \in \mathcal{P}(w)$ we define its complement $\I^\mathrm{C} = \{1,\ldots,m\} \setminus \I$.
The number of sets in $\mathcal{P}(w)$ is equal to $N = 2^{c}$, where $c$ is the number of degenerate indices.

Let $\bar w$ be a feasible point of the MPCC~\eqref{eq:mpcc_ocp}.
Using the partitions $\I \in \mathcal P(\bar w)$, we define the $N$ so-called branch NLPs, denoted by $\BNLP_\I$, obtained by fixing the corresponding branches of the L-shaped complementarity set:
\begin{mini!}[2]
	{\substack{w\in \R^{n}}}{f(w)\label{eq:bnlp_obj}}
	{\label{eq:bnlp}}{}
	\addConstraint{h(w)}{=0 \label{eq:bnlp_eq}}
	\addConstraint{g(w)}{\leq0 \label{eq:bnlp_ineq}}
	\addConstraint{G_i(w) =0,\ H_i(w) }{\geq 0,\quad \forall i \in \I(\bar{w}) \label{eq:bnlp_comp1}}
	\addConstraint{G_i(w) \geq 0,\ H_i(w)} {= 0,\quad \forall i \in \I^\mathrm{C}(\bar{w}). \label{eq:bnlp_comp2}}
\end{mini!}
Another useful NLP associated to the MPCC~\eqref{eq:mpcc_ocp} is the so-called relaxed NLP (RNLP), which reads as:
\begin{mini!}[2]
	{\substack{w\in \R^{n}}}{f(w)\label{eq:rnlp_obj}}
	{\label{eq:rnlp}}{}
	\addConstraint{h(w)}{=0 \label{eq:rnlp_eq}}
	\addConstraint{g(w)}{\leq0 \label{eq:rnlp_ineq}}
	\addConstraint{G_i(w) =0,\ H_i(w) }{\geq 0,\quad \forall i \in \I_{0+}(\bar{w}) \label{eq:rnlp_comp1}}
	\addConstraint{G_i(w) \geq 0,\ H_i(w)} {= 0,\quad \forall i \in \I_{+0}(\bar{w}) \label{eq:rnlp_comp2}}
	\addConstraint{G_i(w) \geq 0,\ H_i(w)} {\geq 0,\quad \forall i \in \I_{00}(\bar{w}). \label{eq:rnlp_bi}}
\end{mini!} 
Observe that if $\I_{00}(\bar{w}) = \emptyset$, then there is only one BNLP and it coincides with the RNLP.
The point $\bar{w}$ is a local minimizer of the MPCC if and only if it is a local minimizer of every $\BNLP_\I$~\cite{Luo1996,Scheel2000}. 
Therefore, these auxiliary NLPs provide a natural framework for defining tailored MPCC notions.
\begin{definition}
	The MPCC \eqref{eq:mpcc_ocp} is said to satisfy MPCC--LICQ at a feasible point $\bar w$ if the corresponding RNLP associated with $\bar w$ satisfies the LICQ at the same point $\bar w$.
\end{definition}
\begin{definition}
	The MPCC--Lagrangian is the \textit{standard} Lagrangian for the RNLP, and reads as:
	\begin{align*}
		\begin{split}
		\mathcal{L}(w,\lambda,\mu,\xi,\nu) &= f(w) + \lambda^\top h(w) + \mu^\top g(w) \\
										   &- \xi^\top G(w)- \nu^\top H(w),	
		\end{split}
	\end{align*}
	with the MPCC--Lagrange multipliers $\lambda \in \R^{m_h}$, $\mu \in \R^{m_g}$, $\xi \in \R^{m}$ and $\nu \in \R^{m}$.
\end{definition}
Note that the MPCC--Lagrangian $\mathcal L(\cdot)$ differs from the standard Lagrangian $L(\cdot)$ for the NLP reformulation of the MPCC \eqref{eq:mpcc_nlp}, by omitting the terms $G_i(w)H_i(w)$ and their multipliers.
The MPCC--Lagrangian coincides with the usual Lagrangian of the $\BNLP_\I$ and is independent of the index~$\I$.

By applying the KKT conditions to the RNLP~\eqref{eq:rnlp}, we obtain the so-called strong stationarity (S-stationarity) concept.
In particular, a feasible point $\bar{w}$ is called S-stationary if there exist Lagrange multipliers $\bar{\lambda},\bar{\mu},\bar{\xi}$ and $\bar{\nu}$ such that:
\begin{align*}
	&\nabla_{w} \mathcal{L}(\bar{w},\bar{\lambda},\bar{\mu},\bar{\xi}, \bar{\nu}) = 0,\\
	&h(\bar{w}) = 0,\\
	&0 \leq \bar{\mu} \perp -g(\bar{w}) \geq 0,\\
	& G_i(\bar{w}) = 0,\; \bar{\xi}_i \in \R,\; H_i(\bar{w}) \geq0,\; \bar{\nu}_i = 0,\;  & \forall i \in 	\I_{0+}(\bar{w}),\\
	& G_i(\bar{w}) \geq 0,\; \bar{\xi}_i = 0,\; H_i(\bar{w}) = 0,\; \bar{\nu}_i \in  \R,\;  & \forall i \in 	\I_{+0}(\bar{w}),\\
	& G_i(\bar{w}) =0,\; \bar \xi_i \geq 0, \; & \forall i \in \mathcal{I}_{00}(\bar{w}),\\
	& H_i(\bar{w}) =0,\; \bar \nu_i \geq 0, \; &\forall i \in 	\mathcal{I}_{00}(\bar{w}).
\end{align*}
The last two conditions arise from complementarity slackness, but are fixed because the active sets are known at $\bar w$.
Weaker stationarity concepts can be derived from alternative auxiliary NLPs and from tools of nonsmooth calculus~\cite{Scheel2000}.
They are distinguished by weaker sign conditions on $\bar\xi_i$ and $\bar\nu_i$ for $i \in \I_{00}(\bar w)$, and are given as follows:
\begin{itemize}
	\item Weak Stationarity (W-stationarity): if $\bar{\xi}_i \in \R$ and $\bar{\nu}_i \in \R$ for all $i \in \mathcal{I}_{00}(\bar{w})$.
	\item Abadie Stationarity (A-stationarity): if $\bar{\xi}_i \geq 0$ or $\bar{\nu}_i \geq0$ for all $i \in \mathcal{I}_{00}(\bar{w})$.
	\item Clarke Stationarity (C-stationarity):  if $\bar{\xi}_i\bar{\nu}_i \geq0$ for all $i \in \mathcal{I}_{00}(\bar{w})$.
	\item Mordukhovich Stationarity (M-stationarity):  if either $\bar{\xi}_i >0$ and $\bar{\nu}_i >0$ or $\bar{\xi}_i\bar{\nu}_i =0$ for all $i \in \mathcal{I}_{00}(\bar{w})$.
\end{itemize}
The stationarity concepts are ordered according to decreasing strength as follows: S, M, C, A and W.
If $\I_{00} = \emptyset$, then all multiplier-based stationarity concepts coincide and they reduce to S-stationarity.
Except for S-stationarity, all other stationarity concepts may admit first-order descent directions~\cite{Scheel2000,Kim2020}, even if MPCC--LICQ holds.
A stationarity concept, which does not suffer from this is the so-called B-stationarity, which we recall next.
We consider a definition due to Hu and Ralph~\cite{Hu2002}, formulated in terms of BNLPs.
For other equivalent definitions see, e.g.,~\cite[Sec. 2.1]{Kim2020}.
\begin{definition}\label{def:b_stat_bnlp}
	A feasible point $\bar{w} \in \Omega$ of the MPCC~\eqref{eq:mpcc_ocp} is a piecewise stationary or B-stationary point of the MPCC~\eqref{eq:mpcc_ocp} if $\bar{w}$ is a stationary point of the $\BNLP_\I$ for each $\I \in \mathcal{P}(\bar{w})$, that is for each partition, there exists Lagrange multipliers $\bar{\lambda}^{\I},\, \bar{\mu}^{\I},\, \bar{\xi}^{\I}$ and $\bar{\nu}^{\I}$ such that:
	\begin{subequations}\label{eq:kkt_bnlp}
		\begin{align}
			&\nabla_{w} \mathcal{L}(\bar{w},\bar{\lambda}^\I,\bar{\mu}^\I,\bar{\nu}^\I,\bar{\xi}^\I) = 0,\\
			&h(\bar{w}) = 0,\\
			&0 \leq \bar{\mu}^\I \perp -g(\bar{w}) \geq 0,\\
			&G_i(\bar w) = 0,  & &\forall i \in \I,\\
			&H_i (\bar w) = 0  & &\forall i \in \I^{\mathrm{C}},\\
			&0 \leq  H_i(\bar{w}) \perp \bar{\nu}_i^\I \geq 0,  & &\forall i \in \I,\\
			&0 \leq G_i(\bar{w}) \perp \bar{\xi}_i^\I \geq 0,   & &\forall i \in \I^{\mathrm{C}}.
		\end{align}
	\end{subequations}
\end{definition}
Under MPCC--LICQ, a B-stationary point $\bar w$ has the same uniquely determined multiplier vector for every branch NLP $\BNLP_{\I}$ with $\I \in \mathcal P(\bar w)$~\cite[Proposition 4.3.5]{Luo1996}.
As a result, B-stationarity can be checked by verifying the S-stationarity conditions for a single branch NLP, which removes the combinatorial complexity.
Conversely, if $\bar{w}$ is an S-stationary point of the MPCC~\eqref{eq:mpcc_ocp}, then it is also B-stationary~\cite[Theorem~4]{Scheel2000}.

In summary, B-stationarity is the tightest stationarity concept for MPCCs, and under MPCC-LICQ it is equivalent to S-stationarity.
This is the setting considered in this paper.

\subsubsection{Second-order optimality conditions}
The second-order sufficient conditions for MPCCs are also defined in terms of the RNLP and BNLPs.
They are formulated at an S-stationary point so that a single set of multipliers is valid for all BNLPs.
The strong MPCC critical cone is obtained as the union of the usual or strong critical cones of all $\BNLP_{\I}$ at $\bar w$~\cite{Luo1996}:
\begin{align*}
	&\mathcal{C}_{\mathrm{MPCC}}^{\mathrm{S}}(\bar{w}) \!=\! \!\!\bigcup_{\I \in \mathcal P(\bar w)}\!\!
	\Big\{ d\in\R^n \ \Big|\
	 \nabla h(\bar{w})^\top d = 0,\\
	& \nabla g_{i}(\bar{w})^\top d = 0, \ \forall i \in \A_+(\bar{w},\bar{\mu}),\\ 
	& \nabla G_{i}(\bar{w})^\top d = 0, \ \forall i \in \I \cup \{ j \in \I^{\mathrm{C}} \mid \bar \xi_j >0 \} ,\\ 
	& \nabla H_{i}(\bar{w})^\top d = 0, \ \forall i \in \I^\mathrm{C} \cup \{ j \in \I \mid \bar \nu_j >0 \}
	\Big\}. 
\end{align*}
\begin{definition}[MPCC--SSOSC]
	The MPCC--SSOSC (or the piecewise SSOSC) holds at S-stationary point $\bar{w}$ with multipliers
	$(\bar\lambda,\bar\mu,\bar\nu,\bar\xi)$ if
	\[
	d^\top \nabla_{ww}^2 \mathcal{L}(\bar{w},\bar\lambda,\bar\mu,\bar\xi,\bar\nu)\, d>0
	\quad \forall  d\in \mathcal{C}^\mathrm{S}_{\mathrm{MPCC}}(\bar{w})\setminus\{0\}.
	\]
\end{definition}
The piecewise (S)SOSC implies the standard (S)SOSC for each BNLP.
It is sufficient for local optimality of an S-stationary point.
A stronger condition is the so-called RNLP-SSOSC, which is the usual SSOSC of the RNLP~\eqref{eq:rnlp}.

Next, we define strict-complementarity type conditions for MPCCs, referred to as upper-level strict complementarity (ULSC)~\cite{Scheel2000}.
In addition, in this paper, we define a new weaker condition, the partial ULSC (PULSC).
\begin{definition}[ULSC]\label{def:ulsc}
	Let $\bar w$ be an S-stationary point of the MPCC~\eqref{eq:mpcc_ocp}, with associated MPCC--Lagrange multipliers $(\bar \lambda, \bar \mu, \bar \xi, \bar \nu)$.
	The upper-level strict complementarity (ULSC) condition holds if, for every $i \in \I_{00}(\bar w)$, the multipliers satisfy the condition:
	$\bar \xi_i > 0$ and $\bar \nu_i > 0$.
\end{definition}
\begin{definition}[PULSC]\label{def:pulsc}
	Let $\bar w$ be an S-stationary point of the MPCC~\eqref{eq:mpcc_ocp}, with associated MPCC--Lagrange multipliers $(\bar \lambda, \bar \mu, \bar \xi, \bar \nu)$.
	The partial upper-level strict complementarity (PULSC) condition holds if, for every $i \in \I_{00}(\bar w)$, the multipliers satisfy the condition:
	$\bar \xi_i > 0$ or $\bar \nu_i > 0$.
\end{definition}
\section{Sequential quadratic programming with complementarity constraints (SQPCC)}~\label{sec:sqpcc}
The SQPCC method was introduced by Scholtes~\cite{Scholtes2004}.
This method is a key ingredient for real-time hybrid MPC algorithms introduced in Sec.~\ref{sec:mpc}.
Here, we recall a recent local convergence result that does not require active-set stabilization, strict complementarity, or ULSC~\cite{Nurkanovic2026}.
Complementarity active-set changes in the SQPCC method are relevant, since they correspond to discovering new switches in real-time hybrid MPC.

In classical MPC, SQP iterations are closely linked to parametric NLP sensitivities and piecewise affine solution approximations, cf. Sec.~\ref{sec:nlp_sesnitivity}.
Due to the violation of constraint qualifications in NLP reformulations of MPCCs, this link is lost when SQP is applied to them. 
This motivates the use of SQPCC, where an analogous relation to the parametric MPCC~\eqref{eq:mpcc_ocp} is preserved, as shown in Sec.~\ref{sec:diff_mpcc}.

The SQP method for standard NLPs is derived from a piecewise linearization of the KKT conditions.
In analogy to SQP, in the SQPCC method, the B-stationarity conditions~\eqref{eq:kkt_bnlp} of the MPCC are structurally linearized, cf.~\cite{Scholtes2004,Nurkanovic2026}.
This corresponds to the B-stationarity conditions of a quadratic program with complementarity constraints (QPCC):
\begin{mini!}[2]
	{\substack{\Delta w \in \R^{n}}}{\nabla f^{k,\top} \Delta w + \frac{1}{2} \Delta w^\top Q^k \Delta w \label{eq:qpcc_obj}}
	{\label{eq:qpcc}}{}
	\addConstraint{h^k + \nabla h^{k,\top} \Delta w + Mx}{=0 \label{eq:qpcc_eq}}
	\addConstraint{g^k +\nabla g^{k,\top} \Delta w}{\leq0 \label{eq:qpcc_ineq}}
	\addConstraint{0 \!\leq \!G^k \!+\! \nabla G^{k,\top} \! \Delta w \!\perp\! H^k \!+\! \nabla H^{k,\top}\! \Delta w }{\!\geq \!0 \label{eq:qpcc_comp},}
\end{mini!}
where $Q^k \approx \nabla_{ww}^2 \mathcal{L}(w^k,\lambda^k,\mu^k,\xi^k,\nu^k)$ is a positive definite approximation of the exact Hessian.
It is important to note that the QPCC uses the Hessian of the MPCC Lagrangian $\mathcal L(\cdot)$, and not the standard Lagrangian $L(\cdot)$ of~\eqref{eq:mpcc_nlp}.
The parameter $x$ is fixed in~\eqref{eq:qpcc_eq}, but we keep it for later reference.

The convergence properties of both the SQP and SQPCC methods depend on the problem functions used to define the quadratic subproblem.
We collect all relevant functions in
\begin{align*}
	\Psi(z) := \Big(\nabla_{w} \mathcal L(w,\lambda,\mu,\xi,\nu),\, h(w),\, -g(w),\, G(w),\, H(w)\Big),
\end{align*}
where $z = (w,\lambda,\mu,\xi,\nu)$ collects all primal-dual variables of the MPCC.
The assumptions for convergence are formulated for $\nabla \Psi^\top(z^{k})$ and its approximation $\nabla \tilde \Psi(z^{k})^\top$, where $Q^k$ replaces the exact Hessian of $\mathcal{L}(\cdot)$ in the QPCC~\eqref{eq:qpcc}, cf. Eq.~\eqref{eq:kappa_omega_sqpcc}.
We discuss the assumptions below the theorem.
\begin{theorem}[Theorem 5.1 in~\cite{Nurkanovic2026}]\label{th:sqpcc_local_convergence}
	Let $f:\R^n\to\R$, $h:\R^n\to\R^{m_h}$, $g:\R^n\to\R^{m_g}$, $H:\R^n \to \R^{m}$, and $G:\R^{n} \to \R^{m}$ be twice continuously differentiable in a neighborhood of a point $\bar{w}\in\R^n$.
	Let $\bar{w}$ be a local solution of problem~\eqref{eq:mpcc_ocp}, satisfying MPCC--LICQ and MPCC--SSOSC with the associated unique Lagrange multiplier $(\bar\lambda,\bar\mu,\bar \xi, \bar\nu )\in \R^{m_h}\times \R^{m_g} \times \R^{m} \times \R^{m}$.
	Further, suppose there exist constants $\omega > 0$ and $\bar{\kappa} < \frac{1}{3\gamma}$, where $\gamma > 0$ is a Lipschitz constant of the BNLPs, and a sequence $\{\kappa^k\}$ with $\kappa^k \in [0,\bar \kappa]$ such that, for all $z^{k}$ and $z$, the following inequalities hold:
	\begin{subequations}\label{eq:kappa_omega_sqpcc}
		\begin{align}
			&\| \nabla \Psi(z)^\top - \nabla \Psi(z^k)^\top \| \leq \omega \| z- z^k \|, \label{eq:omega_cond_sqpcc}\\
			&\| \nabla \Psi(z^k)^\top - \nabla \tilde \Psi(z^k)^\top \| \le \kappa^k. \label{eq:kappa_cond_sqpcc}
		\end{align}
	\end{subequations}
	Then there exist a constant $\varepsilon > 0$ such that, for any starting point $z^0 = (w^0,\lambda^0,\mu^0,\xi^0, \nu^0)\in \B_{\varepsilon}(\bar z)$ and $\bar z = (\bar{w},\bar\lambda,\bar\mu,\bar\nu,\bar \xi)$, there exists a (not necessarily unique) sequence 
	$\{(w^k,\lambda^k,\mu^k,\nu^k,\xi^k)\} \subset \B_{\varepsilon}(\bar z)$ such that for $k=0,1,\dots$, the point $w^{k+1}$ is an S-stationary point of the QPCC~\eqref{eq:qpcc} with a positive definite Hessian approximation $Q^k$, and $(\lambda^{k+1},\mu^{k+1},\nu^{k+1},\xi^{k+1})$ is an associated Lagrange multiplier. 
	Moreover, each such sequence converges to $(\bar{w},\bar\lambda,\bar\mu,\bar\xi,\bar\nu)$, and the rate of convergence is linear, satisfying the contraction inequality
	\begin{align}\label{eq:sqpcc_contraction_estimate}
		\| z^{k+1} - \bar z \| \le \alpha^k \| z^{k} - \bar z \| + \beta \| z^{k+1} - \bar z \|^2,
	\end{align}
	with constants $\alpha^k \in [0,1)$ for all $k$ and $\beta > 0$.
	If, in addition, $\{ \kappa^k \} \to 0$, then $\alpha^k \to 0$ and the convergence rate is superlinear; if $\bar \kappa = 0$, then $\alpha^k = 0$ and the convergence rate is quadratic.
\end{theorem}
This theorem provides a local convergence estimate in~\eqref{eq:sqpcc_contraction_estimate} analogous to Newton/SQP local convergence estimates.
The MPCC-LICQ and MPCC-SSOSC are natural MPCC counterparts of the LICQ and SSOSC assumed in the local convergence of the SQP method~\cite{Deuflhard2011}.
The $\kappa$--$\omega$ conditions~\eqref{eq:kappa_omega_sqpcc} are commonly employed in the analysis of Newton-type methods~\cite{Deuflhard2011}, to which SQPCC also belongs~\cite{Scholtes2004,Nurkanovic2026}.
Condition~\eqref{eq:omega_cond_sqpcc} is a Lipschitz assumption on $\nabla \Phi(z)^\top$ and quantifies the nonlinearity of the problem data.
In turn, the compatibility condition~\eqref{eq:kappa_cond_sqpcc} bounds the error in the QPCC approximations, e.g., for the exact Hessian in the QPCC~\eqref{eq:qpcc} we have $\bar \kappa = 0$ (and quadratic convergence), whereas for a Gauss--Newton approximation usually $\kappa^k > 0$ (and linear convergence).
The smaller the approximation error $\kappa^k > 0$, the faster the linear convergence rate; cf.~\cite{Nurkanovic2026} for details.
Finally, we note that the SQPCC method does not require a global optimizer of the QPCC subproblems; S-stationarity is sufficient.
Under the given assumptions, this can be achieved by local QPCC solvers~\cite{Hall2024,Pozharskiy2026}, and keeping the iterates in $\B_{\varepsilon}(\bar z)$ can be achieved by warm-starting the QPCC solver.

\section{On parametric MPCCs}\label{sec:mpcc_sensitivity}
The main goal of this section is to characterize the continuity and differentiability properties of the solution map of a parametric MPCC in order to properly approximate it and quantify respective errors. 
For the sake of generality, we regard a version of the MPCC~\eqref{eq:mpcc_ocp} where all problem functions depend on a parameter $p \in \R^{n_p}$:
\vspace{-0.15cm}
\begin{mini!}[2]
	{\substack{w\in \R^{n}}}{f(w,p)\label{eq:par_mpcc_obj}}
	{\label{eq:par_mpcc}}{}
	\addConstraint{h(w,p)}{=0 \label{eq:par_mpcc_eq}}
	\addConstraint{g(w,p)}{\leq0 \label{eq:par_mpcc_ineq}}
	\addConstraint{0 \leq G(w,p) \perp H(w,p)}{\geq 0, \label{eq:par_mpcc_comp}}
\end{mini!}
We denote the branch NLPs corresponding to a feasible point $\bar{w}$ with a parameter $\bar{p}$ as $\mathrm{BNLP}_\I(\bar p)$, where $\I \in \mathcal{P}(\bar{w},\bar p)$.
The optimal value function of~\eqref{eq:par_mpcc} is denoted by $V(p)$. 

\subsection{Piecewise differentiability of MPCC solutions}
The parametric MPCC literature has mainly studied the (directional) differentiability of the optimal value function of~\eqref{eq:par_mpcc}~\cite{Hu2002,Guo2014,Izmailov2004}.  
For real-time MPC, however, a key object is the local solution map $z(p)$ itself, since online feedback updates must track a locally optimal primal--dual point as the parameter changes. 
Existing results on this question, e.g., under RNLP-SSOSC and ULSC~\cite[Theorem~11]{Scheel2000} prove $\mathrm{PC}^1$ of $z(p)$.
These assumptions exclude the complementarity active-set changes that are central in hybrid MPC, since they correspond to changes in the switching sequence.

Our results below extend this in the direction relevant for hybrid MPC, but under weaker assumptions.
Theorem~\ref{th:mpcc_sesitivity} establishes a local $\mathrm{PC}^1$ selection of B-stationary points under MPCC--LICQ and the weaker MPCC--SSOSC, and shows local uniqueness already under the weaker PULSC condition (Def.~\ref{def:pulsc}) introduced here. 
In addition, Theorem~\ref{th:mpcc_jump}  identifies when such a local selection must become discontinuous due to complementarity active-set changes. 
Together, these results provide the MPCC counterpart of Theorem~\ref{th:fiacco_nlp_sensitivity} while explicitly allowing complementarity active-set changes.

\begin{theorem}\label{th:mpcc_sesitivity}
	Let $\bar{w}$ be a B-stationary point of~\eqref{eq:par_mpcc} for a parameter $\bar{p}$.
	If MPCC--LICQ and MPCC--SSOSC hold for~\eqref{eq:par_mpcc} at $\bar{w}$ and $\bar{p}$, then:
	\begin{enumerate}[(a)]
		\item $\bar{w}$ is an isolated local optimal solution of~\eqref{eq:par_mpcc} for $p = \bar p$, with $(\bar \lambda,\bar \mu,\bar \nu,\bar \xi)$ being the unique MPCC multiplier vector associated with $\bar{w}$;
		\item there exist a neighborhood $\B_{\delta}(\bar{p})$ of $\bar{p}$,
		such that $w(p)$ is a B-stationary point of~\eqref{eq:par_mpcc}, and local minimizer, with associated unique multipliers $(\lambda(p), \mu(p), \nu(p), \xi(p))$ for $p \in \B_{\delta}(\bar{p})$; and the function $z(\cdot)$ is piecewise continuously differentiable ($\mathrm{PC}^1$) in $p \in\B_{\delta}(\bar{p})$;
		\item if in addition the MPCC--Lagrange multipliers satisfy the PULSC condition at $\bar{p}$ (cf. Def.~\ref{def:pulsc}), then the selection $z(p)$ is locally unique;
		\item the optimal value function $V(p)$ is piecewise smooth near $\bar{p}$, and strictly differentiable at $\bar p$, with $\nabla V(\bar p) = \nabla_p \mathcal{L}(\bar w,\bar \lambda,\bar \mu,\bar \nu, \bar \xi, \bar p)$;
		\item in any direction $v \neq 0$, the one-sided directional derivative $D_v z(p)$ exists at $\bar p$.
	\end{enumerate}
\end{theorem}
\textit{Proof.}
Result (a) is 1) of \cite[Theorem 1]{Hu2002}. 
For part (b), the MPCC--LICQ and MPCC--SSOSC imply the usual LICQ and SSOSC at $\bar w$ for each parametric $\BNLP_{\I}(\bar p)$, and $\I \in \mathcal P(\bar w,\bar p)$. 
It follows from Theorem~\ref{th:fiacco_nlp_sensitivity}~(a) that $\bar w$ is an isolated locally optimal solution of each $\BNLP_\I(\bar p)$. 
Furthermore, it follows from applying Theorem~\ref{th:fiacco_nlp_sensitivity}~(b) to each $\BNLP_\I(p)$ for $\I \in \mathcal P(\bar w,\bar p)$, that there exists a neighborhood $\B_{\delta^{\I}}(\bar p)$ of $\bar p$ such that the optimization problem ${\BNLP}_{\I}(p)$ has a locally unique optimal solution $z^{\I}(p)$ which is $\mathrm{PC}^1$ for $p \in \B_{\delta^{\I}}(\bar p)$.

Next, we restrict the set of valid parameters $p$ such that certain complementarity active sets remain unchanged.
By applying Corollary~\ref{lem:active_set_stabilization} to each $\BNLP_{\I}(\bar p)$, it follows that there exists a constant $\rho^{\I} > 0$ such that, for all $p \in \B_{\delta^{\I}}(\bar p) \cap \mathcal B_{\rho^{\I}}(\bar p)$, the following assertion holds: all inactive constraints $G_i(w) > 0$, $i \in \I_{+0}(\bar w) \cap \I^{\mathrm C}$, and $H_i(w) > 0$, $i \in \I_{0+}(\bar w) \cap \I$, in $\BNLP_{\I}(\bar p)$ remain inactive for $w^{\I}(p)$ and all $p \in \B_{\rho^{\I}}(\bar p)$.
Now for $p \in \B_{\rho}(\bar p)$, with $\rho = \min_{\I \in \mathcal{P}(\bar w, \bar p)} \rho^\I$, all nondegenerate indices $i \in \{1,\ldots,m\} \setminus \I_{00}(\bar w)$ remain nondegenerate for all $w^\I(p)$, i.e., $\I_{0+}(\bar w) = \I_{0+}(w^\I(p))$ and 
$\I_{+0}(\bar w) = \I_{+0}(w^\I(p))$.

Next, we discuss the degenerate pairs $j\in \I_{00}(\bar w)$, which in the $\BNLP_\I$ yield (weakly) active inequality constraints. 
For all $i \in \I_{00}(\bar w)$ with additionally $\bar\xi_i >0 $ and $\bar\nu _i >0$, it means that in each constraint $G_i(w) \geq 0, i \in \I^\mathrm{C} \cap \I_{00}(\bar{w})$ and $H_i(w) \geq 0, i \in \I \cap \I_{00}(\bar{w})$ is strictly active with $G_i(\bar{w}) = 0$ and $H_i(\bar w) = 0$, and it follows from Corollary~\ref{lem:active_set_stabilization} that they remain strictly active for $w^{\I}(p)$ and $p\in \B_{\rho^{\I}}(\bar p)$. 

It remains to discuss degenerate pairs $j \in \I_{00}(\bar{w})$ that satisfy $G_j(\bar w) =  H_j(\bar w) = 0$, with $\bar \xi_j = 0$ and/or $\bar \nu_j = 0$. 
These constraints result in weakly active constraints in the BNLPs, and stay either active or become inactive by making $j$ now a nondegenerate pair.

Let $\I_a$ and $\I_b$ denote two branches differing only at index $j$.
Then in $\BNLP_{\I_a}(\bar{p})$ we have:
\begin{align}\label{eq:branch_parametric_a}
	G_j(w) = 0, \ H_j(w) \geq 0,  j \in  \I_{a},
\end{align}
and in $\BNLP_{\I_b}(\bar{p})$: 
\begin{align}\label{eq:branch_parametric_b}
	G_j(w) \geq 0, \ H_j(w) = 0,  j \in  \I_{b}^\mathrm{C},
\end{align}
We discuss now three distinct possible cases.

\textbf{Case I:} $\bar \xi_j = 0$ and $\bar \nu_j > 0$. 
Consider the equality constraint in $\BNLP_{\I_a}(\bar p)$ given by~\eqref{eq:branch_parametric_a}, 
for which at $p = \bar p$, it holds that $G_j(w^{\I_a}(\bar p)) = 0$ and $\xi_j^{\I_a}(\bar p) = 0$.
In any case $\nu_j^{\I_a}(p)\geq 0$, since it is a multiplier of an inequality constraint in~\eqref{eq:branch_parametric_a}.
For $p \neq \bar p$ there are two possible outcomes.

First, for $p \in \B_{\delta^{\I_a}}(\bar p) \cap \B_{\delta^{\I_b}}(\bar p)$, it may occur that $\xi_j^{\I_a}(p) \geq 0$ (possibly after shrinking $\delta^{\I_a}$ and $\delta^{\I_b}$ if necessary).
In this case, it follows that $G_j(w^{\I_a}(p)) = 0$ for $j \in \I_a$ and $G_j(w^{\I_b}(p)) = 0$ for $j \in \I_b^{\mathrm C}$, and the index $j$ remains degenerate, and the BNLP solutions $w^{\I_a}(p) = w^{\I_b}(p)$ are a B-stationary point. 

Second, suppose that $G_j(w^{\I_a}(\bar p)) = 0$ but $\xi_j^{\I_a}(p) < 0$ for $p \neq \bar p$.
Then $w^{\I_a}(p)$ is no longer a B-stationary point of the MPCC as it violated the conditions of S-stationarity.
Consider instead the branch defined in~\eqref{eq:branch_parametric_b}, where $G_j(w^{\I_a}(\bar p)) = 0$ is now a weakly active constraint, because we assumed for Case I that $\bar \xi_j = 0$.
Suppose that $w^{\I_a}(p) = w^{\I_b}(p)$ is a KKT point of $\BNLP_{\I_b}(p)$ with $G_j(w^{\I_b}(p)) = 0$ for $p \in \B_{\delta^{\I_a}}(\bar p) \cap \B_{{\delta^{\I_b}}}(\bar p)$. 
The two problems have the same active set and because of MPCC-LICQ they must have the same multipliers, cf. \cite[Proposition 4.3.5]{Luo1996}.
But this is impossible, because this would require $\xi_j^{\I_b}(p) \geq 0$, which contradicts the assumption of the case, hence $w^{\I_b}(p)$ is not a KKT point of ${\BNLP}_{\I_a}(p)$.
We conclude that the solution $w^{\I_b}(p)$ must satisfy $G_j(w^{\I_b}\!(p))\!>\!0$.

Thus, on the one hand, $w^{\I_a}(p)$ is not anymore a B-stationary point of the MPCC$(p)$.
On the other hand, in $\BNLP_{\I_b}$ the degenerate index $j\in \I_{00}(\bar w)$ becomes nondegenerate with $j \in \I_{+0}(w(p))$, and for all other $p \in \B_{\min(\rho,\delta^{\I_b})}(\bar p)$, all remaining degenerate indices satisfy the conditions of S-stationarity (and hence B-stationarity).
We have proven that there exists a $\delta \leq \min(\rho,\delta^{\I_b})$ such that, for all $p \in \B_{\delta}(\bar p)$, the point $w^{\I_b}(p)$ is a B-stationary point of the MPCC with active sets $\I_{00}(w^{\I_b}(p)) = \I_{00}(\bar w) \setminus \{j\},\ \I_{+0}(w^{\I_b}(p)) = \I_{+0}(\bar w) \cup \{j\}$ and $\I_{0+}(w^{\I_b}(p)) = \I_{0+}(\bar w)$.

\textbf{Case II:} $\bar \xi_j > 0$ and $\bar \nu_j = 0$.
This case is symmetric to Case I, where $w^{\I_b}(p)$ is not anymore a B-stationary if $j$ becomes nondegenerate, and
$w^{\I_a}(p)$ remains B-stationary with active sets $\I_{00}(w^{\I_a}(p)) = \I_{00}(\bar w) \setminus \{j\}$, $\I_{0+}(w^{\I_a}(p)) = \I_{0+}(\bar w) \cup \{j\}$,
and $\I_{+0}(w^{\I_a}(p)) = \I_{+0}(\bar w)$.

\textbf{Case III:} $\bar \xi_j = 0$ and $\bar \nu_j = 0$. 
If $ \xi_j(p) \geq 0$ and $ \nu_j(p) \geq 0$ then both $w^{\I_a}(p) = w^{\I_b}(p)$ remain B-stationary with $j$ remaining a degenerate index.
Otherwise, by repeating arguments of Cases I and II, we have that $w^{\I_a}(p) \neq w^{\I_b}(p)$ but both remain B-stationary, and we have a branching of the solution.
Yet, both selections, $w^{\I_a}(p)$ and $w^{\I_b}(p)$, satisfy the conditions of Theorem~\ref{th:fiacco_nlp_sensitivity}.

Repeating this argument and intersecting the corresponding $\delta$ neighborhoods of $\bar p$, multiple indices $j$ may leave $\I_{00}$ and enter either $\I_{0+}$ or $\I_{+0}$, and we can construct, not a necessarily unique, $w^\I(p)$ which remains a B-stationary point.
This concludes the proof of the statement (b).

Next, we prove result (c). 
Under the PULSC assumptions, Case III cannot occur and there is no branching of the solution.
Therefore the selection $z(p)$ is locally unique for $p \in \B_{\delta}(\bar p)$.

Result (d) is 2) and 3) of \cite[Theorem 1]{Hu2002}. 
Result~(e) follows from applying Theorem~\ref{th:fiacco_nlp_sensitivity}(e) to the $\BNLP_{\I}(p)$, whose KKT point is also the B-stationary point of the MPCC.
\qed

In Theorem~\ref{th:mpcc_sesitivity}, the local uniqueness of $w(\hat p)$ for any fixed parameter $\hat p$ in result~(a) should not be confused with the local branching of the solution map $z(p)$ described in case~(b).
For any fixed branch $z^j(p)$ and any fixed parameter $\hat p$, the conclusion of~(a) remains valid for $w^j(\hat p)$.
Moreover, if follows from the proof, that if the PULSC condition does not hold, then the number of possible branches of $z(p)$ is equal $2^d$, where $d$ is the number of degenerate multiplier pairs $\bar \xi_i = \bar \nu_i = 0$.
\begin{example}[Nonunique selection]
	Note that in Theorem~\ref{th:mpcc_sesitivity}~(b), despite the MPCC--SSOSC holding, if there exists an index $i \in \I_{00}(\bar w)$ with $\bar\xi_i = \bar\nu_i = 0$, then continuous local selections exist but they are not locally unique at the degenerate point $\bar w$.
	This phenomenon was already observed in~\cite[Example~10]{Scheel2000}:
		\[
		\min_{w\in \R^2} \ (w_1 - p)^2 + (w_2 - p)^2 \ \mathrm{s.t.}\ 0 \leq w_1 \perp w_2 \geq 0,
		\]
	which has the unique solution $w(p)=(0,0)$ for $p < 0$, but branches into the two solutions $(0,p)$ and $(p,0)$ for $p \geq 0$.
\end{example}
Locally unique selections can be ensured under the ULSC condition and RNLP-SSOSC, cf. \cite[Theorem 11]{Scheel2000}. 
However, this precludes the possibility of active set changes where degenerate indices become nondegenerate.  
Theorem~\ref{th:mpcc_sesitivity}(c) improved this result by showing that branching of the solution map, despite such active-set changes, does not occur if the PULSC condition holds.  
This is illustrated in the next example.

\begin{example}[Unique selection]\label{ex:mpcc_pulsc}
	Consider the MPCC
	\begin{align*}
		\min_{w\in \R^2} \ (w_1 - p + 1)^2 + (w_2 + p + 1)^2 
		\,\, \mathrm{s.t.} \
		0 \le w_1 \perp w_2 \ge 0,
	\end{align*}
	for $p \in [-2,2]$.  
	We consider two BNLPs, denoted by $\BNLP_{\I_a}$ with $w_1 = 0,\, w_2 \geq 0$ and 
	$\BNLP_{\I_b}$ with $w_1 \geq 0,\, w_2 = 0$. 
	The corresponding solutions of these BNLPs are shown in the left and middle plots of Fig.~\ref{fig:mpcc_pulsc}.
	The right plot illustrates the locally unique MPCC selection. 
	The solution map satisfies $\I_{00}(w(p)) = \{1\}$ for $p \in [-1,1]$, and for $p > 1$ we have $\I_{00}(w(p)) = \emptyset$.  
	Since $\nu(p) > 0$ and $\xi(p) = 0$ at $p = 1$, the PULSC condition holds, and no branching occurs.
\end{example}

The previous example also exhibits an active-set change at $p = -1$, where a nondegenerate index becomes degenerate.  
This case is not covered by Theorem~\ref{th:mpcc_sesitivity}.  
Nevertheless, the degenerate pair remains S-stationary because the multipliers $\xi(p)$ and $\nu(p)$ are nonnegative at $p = -1$.  
This does not hold in general.  
In such cases, the solution may exhibit a jump discontinuity, as illustrated in the next example.

\begin{example}[Discontinuous solution map]\label{ex:mpcc_jump}
	Consider the parametric MPCC for $p \in [-2,2]$:
	\begin{align*}
		\min_{w\in \R^2} \ (w_1 - p - 1)^2 + (w_2 + p - 1)^2 
		\,\, \mathrm{s.t.} \
		0 \le w_1 \perp w_2 \ge 0.
	\end{align*}
	Consider two branch problems: $\BNLP_{\I_a}$ with $w_1 = 0,\, w_2 \geq 0$ and $\BNLP_{\I_b}$ with $w_1 \geq 0,\, w_2 = 0$. 
	The corresponding solutions are shown in the left and middle plots of Fig.~\ref{fig:mpcc_jump}.  
	Consider the solution branch of $\BNLP_{\I_a}$: for $p \in [-2,1)$ it is S-stationary for the MPCC.  
	At $p = 1$, the index $i = 1$ becomes degenerate and the multipliers satisfy $\xi(p) < 0$ and $\nu(p) = 0$.  
	Hence, the S-stationarity conditions do not hold, and the solution must jump to a different branch (see the right plot of Fig.~\ref{fig:mpcc_jump}). 
	In contrast to the previous example, no continuous selection exists.  
	For any $p \in (-1,1)$ there exist two isolated selections with nondegenerate indices.
\end{example}
\begin{figure}[t]
	\centering
	\includegraphics[width=0.49\textwidth]{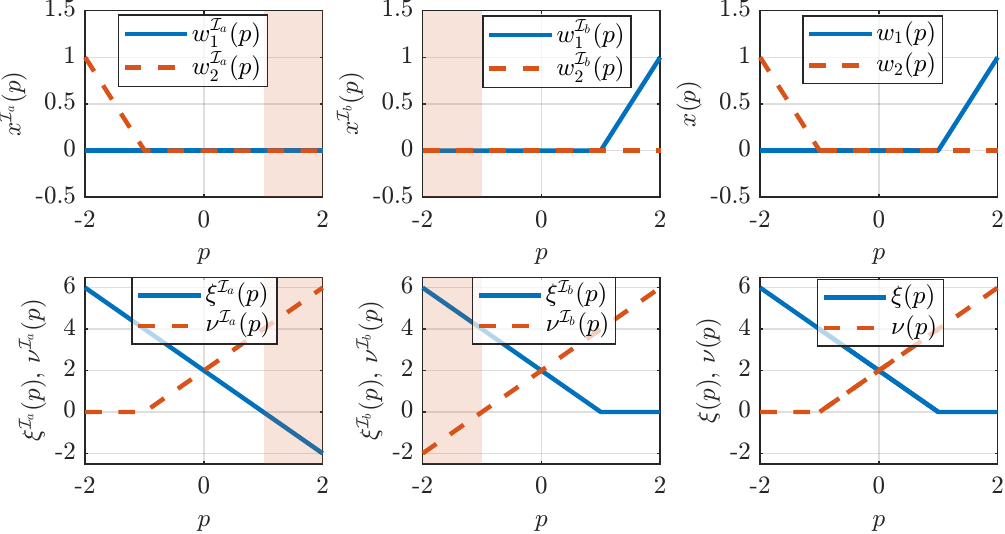}
	\vspace{-0.5cm}
	\caption{
		The left and middle plots show the solution maps of $\BNLP_{\I_a}$ and $\BNLP_{\I_b}$ from Example~\ref{ex:mpcc_pulsc}, respectively.
		The red shaded area indicates where a BNLP solution is not S-stationary.
		The right plot shows the locally unique solution mapping of the MPCC.
	}
	\vspace{-0.4cm}
	\label{fig:mpcc_pulsc}
\end{figure}

So far, Theorem~\ref{th:mpcc_sesitivity} covers active-set changes of the inequality constraints and allows degenerate complementarity pairs $G_i(\bar w) = H_i(\bar w) = 0$ to become nondegenerate for $w(p)$, that is, either $G_i(w(p)) > 0$ or $H_i(w(p)) > 0$.
However, it does not cover the case of the previous two examples at $p = -1$, in which a nondegenerate pair becomes degenerate.
This case is covered by the next theorem.

\begin{theorem}[Solution discontinuity due to active-set changes]\label{th:mpcc_jump}
	Let $\bar w$ be a B-stationary point of~\eqref{eq:par_mpcc} for a parameter $\bar p$.
	Assume MPCC--LICQ and MPCC--SSOSC hold at $\bar w$ and $\bar p$.
	Given a direction $v \in \R^{n_p}$, regard the line segment
	\(
	\mathcal S(\bar p, \bar p + v) := \{ \bar p + \tau v : \tau\in[0,1]\}
	\)
	and let $z^\I(p)$ be a solution to a $\BNLP_\I$ such that $w^\I(\bar p)$ is B-stationary.
	Suppose that there exists an index $i\in \I_{0+}(\bar w)$ and a constant $0<\tau_s< 1$ such that
	$
	i\in \I_{00}(w^{\I}(p(\tau))),\, \text{for all } \tau \in[\tau_s,1],
	$
	provided $\| v\|$ is sufficiently small.
	Then exactly one of the following cases occurs:
	\begin{enumerate}[(a)]
		\item If $\xi_i^{\I}(p(\tau_s))\geq 0,\, i\in \I_{0+}(\bar w)$
		for all $\tau\in[\tau_s,1]$, then for $z(p) := z^{\I}(p)$ remains $\mathrm{PC}^1$, and $w^\I(p)$ is B-stationary for $p \in \mathcal S(\bar p, \bar p + v)$.
		
		\item If $\xi_i^{\I}(p(\tau_s))<0,\, i \in \I_{0+}(\bar w)$,
		then $w^{\I}(p(\tau))$ is not B-stationary for any $\tau \in[\tau_s,1]$.
		Moreover, if there exists a B-stationary point $\hat x$ at $p_s := p(\tau_s)$ satisfying MPCC--SSOSC and MPCC--LICQ, then there exists a solution mapping $z(p)$ that is discontinuous at $p_s$, and  $w(p)$ is B-stationary for all $p \in \mathcal S(\bar p, \bar p + v)$.
	\end{enumerate}
	The symmetric case $i \in \I_{+0}(\bar w)$ with $\nu_i$ in place of $\xi_i$ is treated analogously.
\end{theorem}
\textit{Proof.}
It was already shown in Theorem~\ref{th:mpcc_sesitivity} that, for all other active-set changes, the solution mapping $z(p)$ is $\mathrm{PC}^1$ and $w(p)$ is B-stationary for all $p \in \mathcal B_{\delta}(\bar p)$.
By Corollary~\ref{lem:active_set_stabilization}, shrinking $\delta > 0$ if necessary. W.l.o.g. we assume that all other nondegenerate indices except $i$ remain nondegenerate.

The multiplier $\nu_i(p)$ satisfies $\nu_i(p) \geq 0$, since it corresponds to an inequality constraint in all $\BNLP_\I$.
In case~(a), we have $G_i(w^{\I}(p(\tau_s))) = 0$, $H_i(w^{\I}(p(\tau_s))) = 0$, and $\xi_i(p(\tau)) \geq 0$, with $p(\tau) = \bar p + \tau v$.
Hence, the only additional active-set change is that the index $i$ becomes degenerate, which does not violate the conditions of S-stationarity.

The problem $\BNLP_\I$ satisfies the conditions of Theorem~\ref{th:fiacco_nlp_sensitivity}, and by assumption $p(\tau) \in \mathcal B_{\delta^{\I}}(\bar p)$ for all $\tau \in [\tau_s, 1]$.
Therefore, $w^{\I}(p(\tau))$ remains B-stationary and $z^{\I}(p(\tau))$ is $\mathrm{PC}^1$ for all $\tau \in [0, 1]$, which concludes part~(a).

On the other hand, in case~(b), if $\xi_i^{\I}(p(\tau_s)) < 0$, then by continuity this inequality persists for $p(\tau) = \bar p + \tau v$ for all $\tau \in [\tau_s,1]$ provided $\| v \|$ is sufficiently small.
Consequently, the condition of S-stationarity is violated.
This holds for all $\I \in \mathcal P(\bar w, \bar p)$, since nondegenerate index pairs are identical for all BNLPs.
Since we have assumed the existence of a B-stationary point $\hat w := \hat w(p_s)$ satisfying the conditions of Theorem~\ref{th:mpcc_sesitivity},
we define the function:
\begin{align*}
	z(p(\tau)) :=
	\begin{cases}
		z^{\I}(p(\tau)), & \tau \in [0, \tau_s),\\
		\hat z(p(\tau)), & \tau \in [\tau_s, 1],
	\end{cases}
\end{align*}
such that its primal component $w(\tau)$ is B-stationary for all $\tau\in [0,1]$.

Moreover, since $w^{\I}(p(\tau_s))$ is not S-stationary, it necessarily follows that $\I_{+0}(\bar w) \neq \I_{+0}(\hat w)$ and $\I_{0+}(\bar w) \neq \I_{0+}(\hat w)$.
That is, the nondegenerate index sets do not coincide, which implies that $z(p(\tau))$ is discontinuous at $\tau = \tau_s$.
\qed 

\begin{figure}[t]
	\centering
	\includegraphics[width=0.49\textwidth]{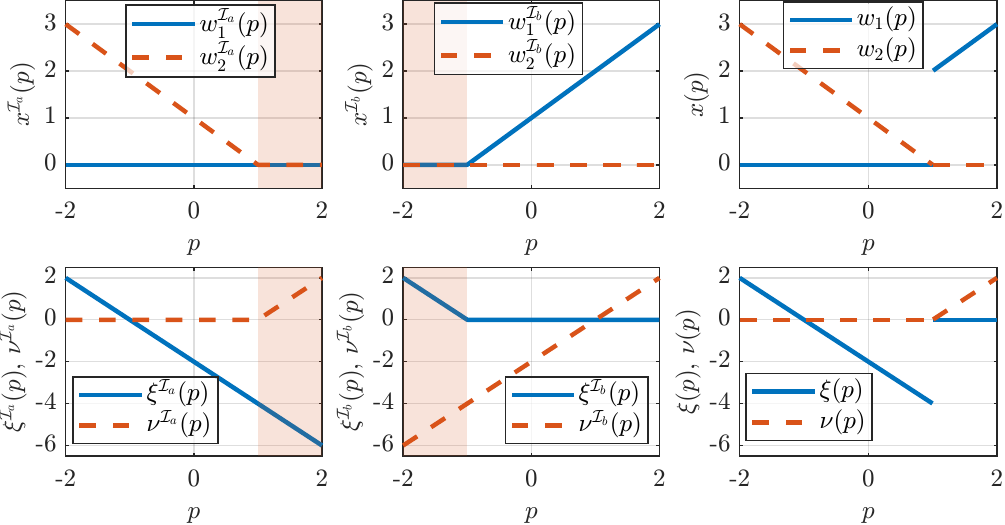}
	\vspace{-0.5cm}
	\caption{
		The left and middle plots show the solution maps of $\BNLP_{\I_a}$ and $\BNLP_{\I_b}$ from Example~\ref{ex:mpcc_jump}, respectively.
		The red shaded area indicates where a BNLP solution is not S-stationary.
		The right plot shows the discontinuous solution mapping of the MPCC.
		For $p \in (-1,1)$ the MPCC has two isolated local minimizers, and the right plot depicts one possible selection, which tracks $\I_a$, as long as possible.
	}
	\vspace{-0.4cm}
	\label{fig:mpcc_jump}
\end{figure}
Jumps and kinks in the solution map $z(p)$ occur whenever a switching event takes place, that is, whenever a complementarity active-set change in the DCS happen~\cite{Nurkanovic2023f}.

We show in Sec.~\ref{sec:path_follow_example}, that QP predictors used in classical MPC are not sufficient to capture such branch changes in the MPCC.
The QP may become infeasible or track a non-optimal branch.
By contrast, a QPCC preserves the complementarity structure under linearization and can therefore capture changes of branch NLPs, i.e., complementarity active-set changes and the associated switching sequence.
Hence, our QPCC-based hybrid real-time MPC algorithms are able to approximate kinks and jumps.

\subsection{Computing directional derivatives via QPCCs}\label{sec:diff_mpcc}
Theorem~\ref{th:mpcc_sesitivity} establishes conditions under which the solution map $z(p)$ is continuous and the directional derivatives $D_v z(\bar p)$ exist in all directions $v$.  
Similarly, in Theorem~\ref{th:mpcc_jump}~(a), directional derivatives exist in all directions since continuity of the solution map is preserved, while in Theorem~\ref{th:mpcc_jump}~(b) they exist in all directions $v$ that do not lead to a jump discontinuity.
Beyond existence, it is also useful to compute these derivatives. 
They are used in \textit{differentiable optimization}, e.g., in optimization layers in neural networks and policy optimization~\cite{Zuliani2025}.  

In analogy to NLP-SQP case~\cite[Theorem 3.6]{Diehl2001}, we show how to compute the directional derivatives of the MPCC solution map via suitable SQPCC step.
For the sake of brevity, we consider the linear parameter case as in~\eqref{eq:mpcc_ocp}, and omit the derivation of the QPCC for the generic MPCC~\eqref{eq:par_mpcc}, which can be developed analogously to, e.g.,~\cite[Eq.~(14)]{Jittorntrum1984} and~\cite[Theorem~1]{Ralph1995}, by exploiting the piecewise structure of MPCCs.

\begin{proposition}\label{prop:directional_derivaies}
	Let $\bar w$ be a B-stationary point of~\eqref{eq:mpcc_ocp} for the base parameter $x = \bar x + \tau v$, with $\tau = 0$.  
	Suppose the assumptions of Theorem~\ref{th:mpcc_sesitivity} or Theorem~\ref{th:mpcc_jump}~(a) hold.  
	Consider the QPCC~\eqref{eq:qpcc} with the exact Hessian, applied to~\eqref{eq:mpcc_ocp} at the linearization point $\bar z$, and the parameter $x = \bar x + \tau v$, and with a sufficiently small $\tau > 0$.  
	Then there exists an S-stationary point $\Delta w$ of this QPCC, with multipliers $(\lambda,\mu,\xi,\nu)$, from which we define:
	\[
	\Delta z := (\Delta w,\lambda-\bar\lambda,\mu-\bar\mu,\xi-\bar\xi,\nu-\bar\nu),
	\]
	and it holds that:
	\begin{align*}
		D_v z(\bar x) = \frac{1}{\tau}\,\Delta z
		= \lim_{\tau\to 0,\, \tau>0} \frac{z(\bar x + \tau v) - z(\bar x)}{\tau}.
	\end{align*}
\end{proposition}
\color{black}

\textit{Proof.}
It follows from the proofs of Theorems~\ref{th:mpcc_sesitivity} and Theorem~\ref{th:mpcc_jump}, that each B-stationary point $w(x)$ is also a KKT point of a corresponding BNLP.
Next, it was established in the proof of Theorem~\ref{th:sqpcc_local_convergence}(cf.~\cite{Nurkanovic2026}) that an S-stationary point of a QPCC is also a solution of the QP applied to such a BNLP.
Hence, for a sufficiently small $\tau$, a single SQPCC iteration is equivalent to an SQP iteration applied to the BNLP, and the result follows directly from~\cite[Theorem~3.6]{Diehl2001}.
\qed

Clearly, the results hold for Theorem~\ref{th:mpcc_jump}~(b), but with the restriction $\tau \in [0,\tau_s)$.
Note that the QPCC enables the computation of directional derivatives not only when the usual inequality active set changes~\cite{Jittorntrum1984}, but also when complementarity active-set changes occur, e.g., in Figs.~\ref{fig:mpcc_pulsc} and \ref{fig:mpcc_jump}.
Such directional derivatives cannot be computed via QPs.

\section{Path-following via the SQPCC method}\label{sec:path_follow}
In the literature, continuously computing approximations to the solution map of a parametric optimization problem is referred to as path-following~\cite{Diehl2009c}.
This is what essentially all real-time MPC algorithms do, and here we investigate how this is done for MPCCs via the SQPCC method.

In many real-time MPC algorithms, only a single SQP iteration is carried out in the \textit{feedback phase} (after a new state estimate $x_{k+1}$ becomes available)~\cite{Diehl2009c}.
Here, we consider the generalization of this idea to the SQPCC method as a core building block for hybrid real-time MPC algorithms.
Section~\ref{sec:mpc} presents several algorithms based on this principle, which further improve the current linearization point through additional computations in the \textit{preparation phase} (in-between samples, before the next state estimate is available).

In particular, we consider a parameter sequence $\{x_k\}_{k\geq 0}$ in the parametric MPCC~\eqref{eq:mpcc_ocp}, with solutions denoted compactly by $\bar z^{k} = z(x_k)$.  
For each new parameter value $x_{k+1}$, a single QPCC is solved at the linearization point $z^{k}$ to obtain $z^{k+1}$, which is an approximation of $\bar z^{k+1}$:
\begin{mini!}[2]
	{\substack{\Delta w \in \R^{n}}}{\nabla f^{k,\top} \Delta w + \frac{1}{2} \Delta w^\top  Q^k \Delta w \label{eq:qpcc_rti_obj}}
	{\label{eq:qpcc_rti}}{}
	\addConstraint{h^k + \nabla h^{k,\top} \Delta w + M x_{k+1}}{=0 \label{eq:qpcc_rti_eq}}
	\addConstraint{g^k +\nabla g^{k,\top} \Delta w}{\leq0 \label{eq:qpcc_rti_ineq}}
		\addConstraint{0 \!\leq \!G^k \!+\! \nabla G^{k,\top} \! \Delta w \!\perp\! H^k \!+\! \nabla H^{k,\top}\! \Delta w }{\!\geq \!0 \label{eq:qpcc_rti_comp},}
\end{mini!}
Observe that, since only one SQPCC iteration is performed per parameter update, the same index $k$ is used for both the SQPCC iterations and parameters.

\subsection{A tutorial example}\label{sec:path_follow_example}
We first illustrate typical properties of the path-following algorithm that performs a single SQPCC iteration per new parameter value using an example.
Afterwards, we provide a theoretical analysis that formalizes the observed error margins.

\begin{example}
	Consider the nonlinear parametric MPCC where the parameter $x$ enters the problem linearly:
	\begin{mini!}[2]
		{\substack{w \in \R^{7}}}{(w_1+1.5)^4 + \sum_{i=2}^{6} w_i^2 }{\label{eq:patthfollow_mpcc}}{}
		\addConstraint{w_7-x}{=0}
		\addConstraint{-w_7^3 -w_4 - w_6}{= 0 }
		\addConstraint{w_7^3 - w_5 - w_6}{= 0 }
		\addConstraint{w_2+w_3}{= 1 }
		\addConstraint{w_1 \!-\! w_2(1\!+\!0.25w_7^2\!+\!0.15\cos(10w_7)) \!-\! 0.5 w_3 w_7^2}{\!\geq\! 0 }
		\addConstraint{0 \leq w_2 \perp w_4}{\geq 0}
		\addConstraint{0 \leq w_3 \perp w_5}{\geq 0.}
	\end{mini!}
	It can be verified that the closed form solution for $x \leq 0$:  $w_1 = 0.5x^2$, $w_2 = 0$, $w_3 = 1$, $w_4 = -2x^3$, $w_5 = 0$, $w_6 = x^3$,$w_7 = x$; and for
	$x > 0$: $w_1 = 1+0.25x^2+0.15\cos(10x)$, $w_2 = 1$, $w_3 = 0$, $w_4 = 2 x^3$, $w_5 = 0$, $w_6 = -x^3$, and $w_7 = x$.
	We observe that the solution map has a discontinuity at $x = 0$.
	The solution map of $w_1(x)$ is illustrated in Fig.~\ref{fig:control_law}.
\end{example}

First, consider the solution map of the QPCC~\eqref{eq:qpcc_rti} associated with the parametric MPCC~\eqref{eq:patthfollow_mpcc}.  
The QPCC-linearization point is the MPCC solution $z(\bar x)$ at $\bar x = -0.1$.  
Additionally, we also consider the solution map of the QP obtained from the BNLP at $z(\bar x)$.
In an MPC setting, a single SQP step applied to the solution of this BNLP provides a QP tangential predictor.
It is well-known that the solution map of a convex QP is continuous piecewise affine.  
The feasible set of a QPCC is the union of the feasible sets of the corresponding QPs, and by Theorem~\ref{th:mpcc_jump}~(b) the solution map may be discontinuous.  
Accordingly, the QPCC solution map is a piecewise affine, possibly discontinuous, approximation of the MPCC solution.  

\begin{figure}[t!]
	\centering
	\includegraphics[width=0.24\textwidth]{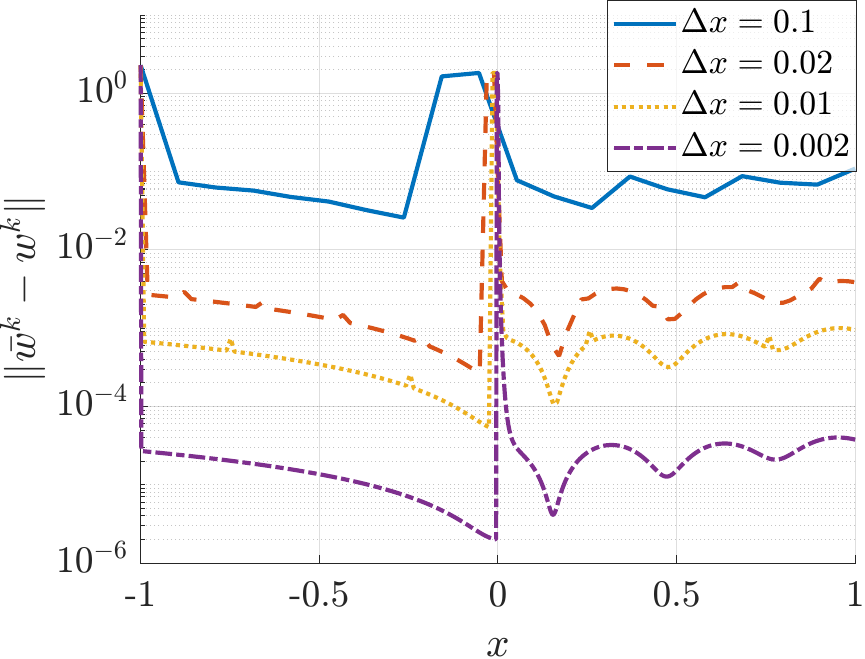}
	\includegraphics[width=0.24\textwidth]{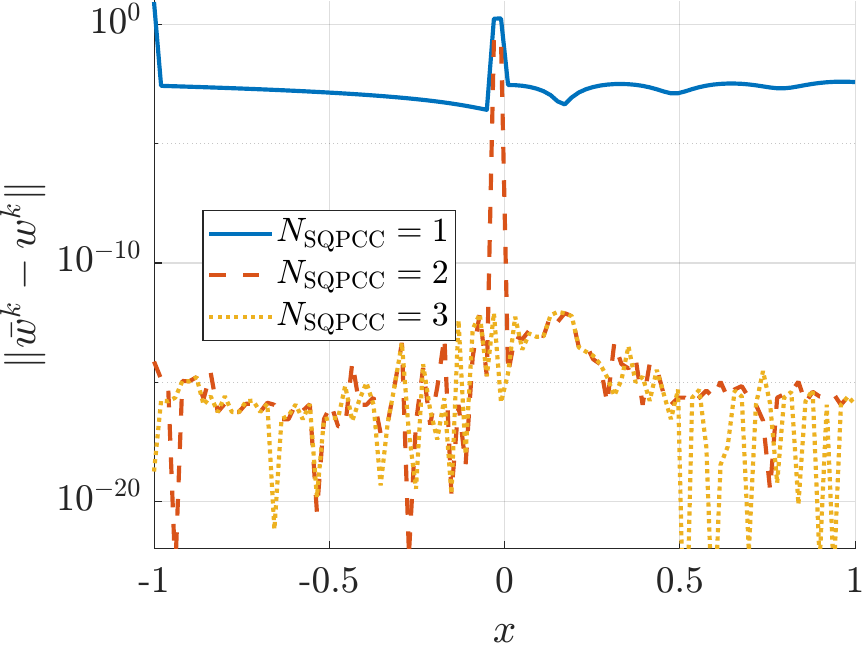}
	\vspace{-0.5cm}
	\caption{
		Tracking error of SQPCC path-following schemes. 
		Left: single-step SQPCC tracking for different parameter increments $\Delta x$. 
		Right: error reduction with multiple SQPCC iterations per parameter.
	}
	\vspace{-0.35cm}
	\label{fig:sampling_time_sqp}
\end{figure}

Fig.~\ref{fig:control_law} illustrates this behavior: as long as the active sets agree for $x < 0$, the QPCC and QP solutions are tangent to the MPCC solution.
In contrast, the QP problem becomes infeasible for parameter values at which the QPCC solution has discovered a new complementarity active set and provided a discontinuous approximation.

Next, we consider the tracking error obtained by performing a single SQPCC iteration per new parameter value $x_k$.  
This behavior is illustrated in the left plot of Fig.~\ref{fig:sampling_time_sqp}.  
Smaller parameter changes $\Delta x$ improve tracking accuracy and lead to faster error decay (around $x =0$) after solution-map discontinuities.
This is formalized in Theorem~\ref{th:path_follow_sqpcc_formal}.  
If several SQPCC iterations are performed per fixed parameter value, e.g., with an exact Hessian in~\eqref{eq:qpcc_rti}, then the quadratic convergence from Theorem~\ref{th:sqpcc_local_convergence} implies rapid error decay, including errors introduced by jumps, as shown in Fig.~\ref{fig:sampling_time_sqp}(right).
Such additional correcting steps for a fixed parameter are the basis for advanced-step type MPC defined in Sec. \ref{sec:asc} and \ref{sec:asrti}.

\subsection{Error analysis}
Next, we formalize the observations from the previous example. 
In particular, it is shown that if the parameter change $\Delta x$ is sufficiently small, the overall error remains bounded over all iterations, since the SQPCC step compensates for the variation of the problem solution. 
This remains true even if the solution exhibits jumps, provided the new solution to be tracked stays within the local convergence region. 

\begin{theorem}[One-step SQPCC error]\label{th:path_follow_sqpcc_formal}
	Consider the parametric MPCC~\eqref{eq:mpcc_ocp} on a parameter set $X\subset\R^{n_x}$ and a parameter sequence
	$\{x_k\}_{k\ge 0}\subset X$.
	For each $k$, let $\bar z^{k} :=  z(x_k)=(\bar w^k,\bar\lambda^k,\bar\mu^k,\bar\xi^k,\bar\nu^k)$ denote a selected
	solution of~\eqref{eq:mpcc_ocp} at $x_k$.
	For each new parameter value $x_{k+1}$, let $z^{k+1}$ be an S-stationary point of the QPCC ~\eqref{eq:qpcc_rti} linearized at $z^k$ and evaluated at $x_{k+1}$ and with a positive definite Hessian approximation $Q^k$.
	Further, assume the following.
	\begin{enumerate}[(a)]
		\item (Regularity).
		There exist sets $X_1,\dots,X_{N_X}\subset X$ with $X=\bigcup_{i=1}^{N_X} X_i$ such that for every $i$ there exist
		$\bar x_i\in X_i$ and $\delta^x_i>0$ with $X_i\subset \B_{\delta^x_i}(\bar x_i)$ and such that the point $\bar w(\bar x_i)$ is a local minimizer of~\eqref{eq:par_mpcc} for $x= \bar x_i$ and satisfies MPCC--LICQ and MPCC--SSOSC at $(\bar w(\bar x_i), \bar x_i)$.
		For each $i$, let $z_i(\cdot)$ denote a $\mathrm{PC}^1$ local solution selection on $X_i$ provided by
		Theorem~\ref{th:mpcc_sesitivity}, and assume that the selection $ z(\cdot)$ satisfies $z(x)=z_i(x)$ for all $x\in X_i$.
		
		\item (Jump size).
		There exists $\bar r_x>0$ such that for all $k$ under consideration,
		\(
		\|x_{k+1}-x_k\|\leq \bar r_x
		\)
		implies that the line segment
		\(
		\mathcal S(x_k,x_{k+1}) := \{(1-\tau)x_k+\tau x_{k+1}:\tau\in[0,1]\}
		\)
		intersects the boundaries $\partial X_i$ in at most one parameter value.
		Moreover, if $x_k,x_{k+1}\in X_i$ for some $i$, then $\bar z(\cdot)$ is continuous on $\mathcal S(x_k,x_{k+1})$ (due to Theorem~\ref{th:mpcc_sesitivity}).

		If $x_k\in X_i$ and $x_{k+1}\in X_j$ with $i\neq j$, then there exists a (necessarily unique) switching parameter
		$x_k^s\in \mathcal S(x_k,x_{k+1})$ at which $ z(\cdot)$ is discontinuous, and this discontinuity is of the type characterized by
		Theorem~\ref{th:mpcc_jump}~(b).
		Define the associated jump size
		\begin{align*}
			\delta_k \;:=\;
			\begin{cases}
				0,  &\textnormal{if $z(\cdot)$ is continuous on $ \mathcal S(x_k,x_{k+1})$},\\[0.2em]
				\displaystyle \Delta z^k,
				 &\textnormal{if $ z(\cdot)$ has a jump at $x_k^s$},
			\end{cases}
		\end{align*}
		with $\Delta z^k := \lim_{\tau\downarrow 0}\|\bar z(x_k^s+\tau(\Delta x_{k})-\bar z(x_k^s-\tau\Delta x_{k})\|$.
		\item (SQPCC local convergence).
		The assumptions of Theorem~\ref{th:sqpcc_local_convergence} hold at $\bar z^{k}$ for all $k$ under consideration, with a
		uniform local convergence radius $\varepsilon>0$ such that the statement of Theorem~\ref{th:sqpcc_local_convergence} applies whenever
		$z^{k}\in\B_{\varepsilon}(\bar z^{k})$.
	\end{enumerate}
	Then there exists a sufficiently small constant $0 < r_z < \varepsilon$ such that, if the initial error satisfies $\| \bar z^0 - z^0 \| \leq r_z$, then there exist constants $\alpha^k \in [0,1)$, $\eta_1^k \geq 0$, and $\beta,\, \eta_2,\, \sigma \geq 0$ (independent of $k$) such that, for all $k$,
	\begin{equation}\label{eq:path_follow_bound}
		\begin{aligned}
		\|\bar z^{k+1}-z^{k+1}\| &\leq \alpha^k\|\bar z^{k}-z^{k}\| + 2\beta\|\bar z^{k}-z^{k}\|^{2} \\
		&+ \eta_1^k\|x_{k+1}-x_{k}\| + \eta_2\|x_{k+1}-x_{k}\|^{2} \\
		&+ \alpha^k \delta_k + \sigma \delta_k^2 .
		\end{aligned}
	\end{equation}
	provided the parameter sequence satisfies, for all $k \geq 0$,
	\begin{align}\label{eq:delta_p_bound}
		\| x_{k+1}-x_{k}\| <		r^x_k := 
		\begin{cases}
			\min \Big(\bar r_x, \frac{\varepsilon - r_z - \delta_k}{\gamma}, \frac{\zeta_k}{2\eta_2}\Big),& \textnormal{if}\ \eta_2 >0, \\
			\min \Big(\bar r_x, \frac{\varepsilon - r_z - \delta_k}{\gamma}, \frac{\upsilon_k}{\eta_1^k}\Big),& \textnormal{if}\ \eta_2 =0.
		\end{cases}		
	\end{align}
	with shorthands $\zeta_k\!:=\! \sqrt{(\eta_1^k)^2\!+\!4 \eta_2 (\rho^k \!-\! \alpha^k \delta_k - \sigma \delta_k^2)) }\!-\!\eta_1^k \! >\!0$,
	$\upsilon^k:= \rho^k-\alpha^k \delta_k - \sigma \delta_k^2>0$, 
	and the jump sizes $\delta_k$ satisfy
	\begin{align}\label{eq:jump_size_bound}
		\delta_k < \begin{cases}
			\min \Big(\varepsilon - r_z, \frac{\sqrt{ (\alpha^k)^2+4\sigma \rho^k}-\alpha^k}{2 \sigma }\Big),& \textnormal{if}\ \sigma >0 \\
			\min \Big(\varepsilon - r_z, \frac{\rho^k}{\alpha^k}\Big),& \textnormal{if}\ \sigma =0 ,\\
		\end{cases}
	\end{align}
	with $\rho^k := (1- \alpha^k r_z - 2\beta r_z)r_z > 0$.
	
	Furthermore, the error remains bounded, i.e., it holds that $\|\bar z^{k} - z^{k}\| \leq r_z$ for all $k \geq 0$.
	
\end{theorem}

\textit{Proof.}
Consider the starting point with its respective error $\| \bar z^0 - z^0 \| < r_z$. 
Solving the QPCC~\eqref{eq:qpcc_rti} with linearization point $z^0$ but parameter value $x_1$ can be interpreted as an SQPCC iteration from the initialization point $z^0$ towards the solution $\bar z^1$. 
However, to use~\eqref{eq:sqpcc_contraction_estimate}, we require that $\| \bar z^1 - z^0 \| \leq \varepsilon$, i.e, it should hold that
\begin{align}\label{eq:sqpcc_next_guess}
	\| \bar z^1 - z^0 \|  \leq \| \bar z^1 - \bar z^0 \| + \| \bar z^0 - z^0 \| \leq \| \bar z^1 - \bar z^0 \| + r_z < \varepsilon.
\end{align}
Next, we bound the term $\| \bar z^1 - \bar z^0 \|$.
Suppose that $x_0 \in X_i$ and $x_1 \in X_j$ with $i \neq j$. 
Let $\Delta x_0 := x_1 - x_0$. 
Then there exists $\tau_s \in (0,1)$ such that
$x(\tau) := x_0 + \tau \Delta x_0 \in X_i$ for $\tau \in [0,\tau_s)$ and, without loss of generality, $x(\tau) \in X_j$ for $\tau \in [\tau_s,1]$.
It follows from Theorem~\ref{th:mpcc_sesitivity} that $z(x(\tau))$ is Lipschitz continuous with constant $\gamma_i$ on $\tau \in [0,\tau_s)$ and with constant $\gamma_j$ on $\tau \in [\tau_s,1]$, but is discontinuous at $\tau_s$. 
From the def. of $\delta^k$, we have $\| \bar z(x(\tau_s^-)) - \bar z(x(\tau_s^+)) \| = \delta_1$, where $\delta_1$ is the size of the discontinuity.
Using this, we have:
\begin{align}\label{eq:delta_u_bound}
	\begin{split}
		&\| \bar z^1 - \bar z^0 \| = \| \bar z (x(1)) - \bar z(x(0)) \| \\
		&\leq \| \bar z (x(1)) \!+\! \bar z (x(\tau_s^-))  \!-\! \bar z (x(\tau_s^-)) \\ 
		&\!+\!\bar z (x(\tau_s^+))  \!-\! \bar z (x(\tau_s^+)) \!-\! \bar z(x(0)) \|\\
		&\leq \gamma_j \| x(1) - x(\tau_s^+)\| + \delta_1 + \gamma_i \| x(0) - x(\tau_s^-)\|\\
		&\leq ((1-\tau_s)\gamma_j + \tau_s \gamma_i )\| x(1) - x(0)\| + \delta_1 \\
		&\leq \underbrace{\max(\gamma_i,\gamma_j)}_{:= \gamma }\|x_1 - x_0\| + \delta_1. 
	\end{split}
\end{align}
Note that if $X_i = X_j$, i.e., no jump occurs and $\delta_1 = 0$, then \eqref{eq:delta_u_bound} reduces to the Lipschitz continuity estimate for $\bar z(x)$ on $X_i$.
Combining this with~\eqref{eq:sqpcc_next_guess} and the requirement that $\| \bar z^1 - z^0 \| \leq \varepsilon$, we obtain the bound:
\begin{align*}
	\|x_1  - x_0 \| \leq r^x_1 := \frac{\varepsilon-r_z-\delta_1}{\gamma}.
\end{align*}	
Since $r_z >0$ can be as small as needed, we in addition require a bound on the jump: $\delta_1 < \varepsilon - r_z$, so that $r^x_1  > 0$.
Hence, we are allowed to use~\eqref{eq:sqpcc_contraction_estimate} and obtain:
\begin{align*}
	\| \bar z^1 - \bar z^0\| \leq \alpha^1 \| \bar z^1  - z^0\| + \beta \| \bar z^1  - z^0\|^2.
\end{align*}
In the two terms on the right-hand side, add and subtract $\bar z^0$, then apply the triangle and Cauchy–Schwartz inequalities together with the estimate~\eqref{eq:delta_u_bound} to obtain:
\begin{align*}
	&\| \bar z^1 - z^1\| \leq \alpha^1 \| \bar z^0  - z^0\| + 2\beta \| \bar z^0 - z^0\|^2\\
	&+ \alpha^1\gamma \| x_1  - x_0\| + 4\beta\gamma^2 \| x_1  - x_0\|^2
	+ \alpha^1 \delta_1 + 4 \beta (\delta_1)^2.
\end{align*}
Define the constants $\eta^k_1 := \alpha^k \gamma$, $\eta_2 : = 4\beta\gamma^2$ and $\sigma:= 4 \beta$.
It follows from Theorem~\ref{th:sqpcc_local_convergence} that $\| \bar z^1 - z^1 \| < \varepsilon$.
To ensure the stronger bound $\|\bar z^1 - z^1\| < r_z$, additional restrictions must be imposed on $\delta_1$ and on the parameter step size $r^x_1$, that is, on $\|x_1 - x_0\| < r^x_1$.
From the previous inequality, we have
\begin{align*}
	0 <  \eta_1^1 r^x_1 + \eta_2 (r^x_1)^2  < \underbrace{(1- \alpha^1 r_z - 2\beta r_z)r_z}_{:= \rho^1} - \alpha^1 \delta_1 + \sigma (\delta_1)^2 .
\end{align*}
This inequality makes sense only if the right-hand side is positive, which yields a bound on $\delta_1$.
Since $r_z > 0$ can be chosen arbitrarily small, it follows that $\rho^1 > 0$.
Thus, we solve the quadratic inequality:
\[
\sigma (\delta_1)^2 + \alpha^1 \delta_1  -  \rho^1  < 0.
\]
If $\sigma = 0$, then $\delta_1 < \frac{\rho^1}{\alpha^1}$.
Otherwise, if $\sigma > 0$, then $\delta_1 < \frac{\sqrt{ (\alpha^1)^2+4\sigma \rho^1}-\alpha^1}{2 \sigma }$.
The r.h.s. is always positive since $\rho^1 > 0$ can be made sufficiently small by shrinking $r_z>0$.
Combining all bounds on $\delta_1$ yields the bound in~\eqref{eq:jump_size_bound} with $k = 1$.

Finally, to obtain a bound on $r^x_1$ we solve the quadratic inequality:
\[
\eta_2 (r^x_1)^2 + \eta_1^1 r^x_1  - \rho^1 + \alpha^1 \delta_1 + \sigma (\delta_1)^2 < 0.
\]
If $\eta_2 = 0$, then $r^x_1 < \frac{\rho-\alpha^1 \delta_1 - \sigma (\delta_1)^2}{\eta_1^1}$.
If $\eta_2 > 0$, then
\(
r^x_1 < \frac{\sqrt{(\eta_1^1)^2 +4\eta_2(\rho-\alpha^1 \delta_1 - \sigma (\delta_1)^2)}-\eta_1^1}{2\eta_2}.
\)
The term under the square root is always positive due to the bound on $\delta_1$ since $\rho-\alpha^1 \delta_1 - \sigma (\delta_1)^2 > 0$ under the computed bounds on $\delta_1$.
Summarizing all computed bounds on $r^x_1$ yields~\eqref{eq:delta_p_bound}.
Applying this inductively for all $k \geq 1$ gives the conclusion of the theorem and completes the proof.
\qed

Theorem~\ref{th:path_follow_sqpcc_formal} makes explicit how the admissible jump sizes $\delta_k$ (Eq.~\eqref{eq:jump_size_bound}) and parameter variations $\Delta x_k:=\|x_{k+1}-x_{k}\|$  (Eq.~\eqref{eq:delta_p_bound}) depend on the local convergence radius $\varepsilon$, the prescribed error bound $r_z$, and SQPCC convergence rate constants in~\eqref{eq:sqpcc_contraction_estimate}, in order to keep the error bounded.
A larger convergence radius $\varepsilon$ and a smaller target error bound $r_z$ allow for larger $\Delta x_k$ and larger admissible jump sizes $\delta_k$.
Conversely, using sufficiently small parameter variations $\Delta x_k$ (which correspond to higher sampling rates in MPC) keeps the error bounded, allows for larger admissible jump sizes $\delta_k$, and lead to faster decay of the jump-induced errors, as observed in Fig.~\ref{fig:sampling_time_sqp} (left).

	Certain complementarity active-set changes correspond to switches in the hybrid system and yield discontinuous solutions $z(x)$ of a discrete-time OCP (cf. Theorem~\ref{th:mpcc_jump}).
	The jump-induced tracking error acts as a bounded disturbance to the closed-loop MPC system. Theorem~\ref{th:path_follow_sqpcc_formal} quantifies its magnitude and provides the conditions under which the hybrid MPC algorithms keep the error uniformly bounded.

\section{Real-time algorithms for hybrid MPC}\label{sec:mpc}
Building on the results of the previous sections on parametric MPCCs and SQPCC-based tracking, we propose several approximate hybrid MPC algorithms. 
These methods generalize established real-time MPC ideas~\cite{Diehl2009c} to hybrid systems.

We consider a sequence of sampling times $\{t_k\}$ and corresponding state estimates $\{x_k\}$.
The goal of real-time MPC algorithms is to compute, at each sampling time, a new solution approximation $z^{k+1} \approx \bar z^{k+1}$ and to apply the corresponding feedback at time $t_{k+1}+t_\Delta$ with as small a computational delay $t_\Delta$ as possible.
To reduce the delay, the algorithms split their computations into a \textit{preparation phase}, for $t \in [t_{k},t_{k+1})$, where all computations that can be carried out before a new state estimate $x_{k+1} $ is available, and a \textit{feedback phase}, for $t \in [t_{k+1},t_{k+1} \!+\! t_{\Delta}]$, which uses the minimal amount of computation required to obtain a good solution approximation at each sampling instant.
In hybrid MPC, this minimal computation is a single QPCC solve.

\subsection{The hybrid real-time iteration (HyRTI)}\label{sec:rti}
The classical real-time iteration (RTI) scheme for smooth nonlinear MPC performs a single SQP step per sampling interval $[t_{k},t_{k+1}]$.
In the hybrid setting, this idea extends to performing a single SQPCC iteration, which yields the Hybrid RTI (HyRTI).
Observe that, in optimal control problems~\eqref{eq:discrete_time_ocp} (both classical and hybrid), the new state estimate $x_{k+1}$ enters the equality constraints~\eqref{eq:mpcc_ocp} linearly.
Therefore, all function and derivative evaluations, depending on $z^k$, and required to set up the QPCC~\eqref{eq:qpcc_rti} can be carried out without knowledge of $x_{k+1}$.
Thus, we have the following computations in the HyRTI.
\begin{itemize}
	\item[] \textbf{Preparation phase:}
	At $t \in [t_{k},t_{k+1})$ construct the QPCC~\eqref{eq:qpcc_rti} at the linearization point $z^k$.
	
	\item[] \textbf{Feedback phase:}
	At $t = t_{k+1}$ update the parameter to $x_{k+1}$, solve QPCC~\eqref{eq:qpcc_rti} and return $u_{k+1}$ at $t = t_{k+1}\!+\!t_{\Delta}$.
\end{itemize}
HyRTI is the one-step SQPCC algorithm with conveniently ordered computations.
Hence, it follows from Theorem~\ref{th:path_follow_sqpcc_formal}, that the HyRTI has a bounded error over all iterates. 
It can handle active-set changes both in the standard inequality but also in the complementarity constraints.
It is suitable for small to medium scale systems with fast dynamics.

\subsection{The hybrid advanced-step controller (HyASC)}\label{sec:asc}
The advanced-step controller for nonlinear MPC introduced in~\cite{Zavala2009} solves in the preparation phase an optimal control problem with a predicted initial state $x_{k+1}^{\mathrm{pred}}$ to convergence. 
In the feedback phase, the active sets at this solution are fixed and, by applying the implicit function theorem, a tangential predictor is computed in the feedback phase~\cite[Section 3]{Zavala2009}.
Just like the QP, the usual tangential predictor cannot capture complementarity active-set changes, cf. Fig~\ref{fig:control_law}.
Instead, we must solve a QPCC in the feedback phase.

We generalize this idea to the hybrid setting and introduce the hybrid advanced-step controller~(HyASC), which requires the following computations.
\begin{itemize}
	\item[] \textbf{Preparation phase:}
	For $t \in [t_{k},t_{k+1})$ perform:
	\begin{enumerate}[1.)]
		\item At $t = t_k$, using the previous state and controls $(x_k,u_k)$, predict the next state $x_{k+1}^{\mathrm{pred}}$.
		\item Over $t \in [t_{k},t_{k+1})$, using the initial guess $z^k$, solve the OCP~\eqref{eq:mpcc_ocp} with the parameter $x_{k+1}^{\mathrm{pred}}$ to convergence and obtain $\bar z^{k+1}_{\mathrm{pred}}$.
		\item Construct the QPCC~\eqref{eq:qpcc_rti} at the lin. point $\bar z^{k+1}_{\mathrm{pred}}$.
	\end{enumerate}
	
	\item[] \textbf{Feedback phase:}
	At $t = t_{k+1}$ update the parameter to $x_{k+1}$, and during $t \in [t_{k+1},t_{k+1}+ t_{\Delta}]$ solve the QPCC~\eqref{eq:qpcc_rti} and return $u_0(x_{k+1})$.
\end{itemize}
The MPCC in step 2.) of the preparation phase can be solved with any common MPCC algorithm~\cite{Kim2020,Nurkanovic2024b}. 
This is the most computationally intensive part, and in its basic form, HyASC is more suitable for systems with slower dynamics.
The main benefit of HyASC is that the full MPCC solves can compensate errors that do not satisfy the conditions of Theorem~\ref{th:path_follow_sqpcc_formal}, provided that $x_{k+1}^{\mathrm{pred}}$ is sufficiently close to $x_{k+1}$. 

Lastly, we discuss the feedback-phase error when a QPCC is solved but the linearization point is $\bar z^{k+1}_{\mathrm{pred}}$ instead of $z^k$.
Provided that $\|\bar z^{k+1}_{\mathrm{pred}} - \bar z^{k+1}\| < r_z < \varepsilon$, and that the remaining assumptions of Theorem~\ref{th:path_follow_sqpcc_formal} hold, the error estimate~\eqref{eq:path_follow_bound} applies.
In particular, the bound reduces to:
\begin{align*}
	\begin{split}
	\|\bar z^{k+1}-z^{k+1}\|
	&\leq
	\eta_1^k \,\|x_{k+1}-x_{k+1}^{\mathrm{pred}}\|
	+\eta_2\,\|x_{k+1}-x_{k+1}^{\mathrm{pred}}\|^{2}\\
	&+\alpha^k \delta_k + \sigma \delta_k^2.
	\end{split}
\end{align*}
In addition, if no complementarity active-set changes happen between $x_{k+1}$ and $x_{k+1}^{\mathrm{pred}}$, then  $\delta_k = 0$, and the bound is even tighter.
Further, using the exact Hessian yields in this case the bound $O(\|x_{k+1}-x_{k+1}^{\mathrm{pred}}\|^{2})$.

In practice, to reduce the feedback time in the HyASC, in the feedback phase one may solve a branch QP corresponding to the MPCC's solution found in the preparation phase.
If the feedback QP introduces a large error, the next preparation-phase solve can compensate for it.

\subsection{The hybrid advanced step RTI (HyAS-RTI)}\label{sec:asrti}
The advanced-step RTI (AS-RTI) combines the ideas of the ASC and RTI methods~\cite{Nurkanovic2019a,Frey2024a}.
In the feedback phase, it solves a QP as in RTI, while in the preparation phase it does not solve the problem to convergence but performs $N^{\mathrm{as}} \geq 1$ iterations of an (inexact) SQP method to improve the linearization point $z^k$ for the feedback QP~\cite{Nurkanovic2019a,Frey2024a}.
These ideas extend naturally to the hybrid setting, yielding the hybrid AS-RTI (HyAS-RTI) method, with computations as follows.
\begin{itemize}
	\item[] \textbf{Preparation phase:}
	For $t \in [t_{k},t_{k+1})$ perform:
	\begin{enumerate}[1.)]
		\item At $t = t_k$, using the previous state and controls $(x_k,u_k)$, predict the next state $x_{k+1}^{\mathrm{pred}}$.
		\item Over $t \in [t_{k},t_{k+1})$, using the initial guess $z^k$, perform $N^{\mathrm{as}}$ iterations~\eqref{eq:qpcc_rti} on the OCP~\eqref{eq:mpcc_ocp} with parameter $x_{k+1}^{\mathrm{pred}}$ to obtain the solution approximation $z^{k+1}_{\mathrm{pred}}$.
		\item Construct the QPCC~\eqref{eq:qpcc_rti} at the point $z^{k+1}_{\mathrm{pred}}$.
	\end{enumerate}
	\item[] \textbf{Feedback phase:}
	At $t =t_{k+1}$ update the parameter to $x_{k+1}$, during  $t \in [t_{k+1},t_{k+1}+ t_{\Delta}]$ solve the QPCC~\eqref{eq:qpcc_rti} and return $u_0(x_{k+1})$.
\end{itemize}

For the sake of brevity, we consider only SQPCC iterations in the preparation phase.
It is straightforward to extend this to inexact SQPCC or zero-order SQPCC variants as in~\cite{Nurkanovic2019a,Frey2024a}.
Of practical interest are variants that do not significantly increase the computational load compared to HyRTI but still reduce the numerical error.
In particular, performing a single SQPCC iteration, i.e., $N^{\mathrm{as}}=1$, in the preparation phase and a QPCC in the feedback phase yields a practical algorithm.

If the conditions of Theorem~\ref{th:path_follow_sqpcc_formal} hold, then a more detailed error analysis analogous to~\cite{Frey2024a} can be derived.

\section{Numerical example}\label{sec:examples}
In this section, we showcase the new hybrid MPC algorithms on a robotic manipulation example.
The example shows that the algorithms are able to discover new contact sequences online, react to disturbances, and handle changing references.

We consider a two-link robot modeling a finger, which is supposed to spin a fixed wheel.
This example is inspired by a similar example from~\cite{Posa2014}, and a full description of the dynamic equations is also available in the example's GitHub repository.

The system has three degrees of freedom $q = (q_1,q_2,\theta)$, namely the two joint angles and the wheel angle.
The state consists of the positions and velocities, with $x = (q,v) \in \R^6$, where $v = \dot q$.
The controls are the two joint torques $u \in \R^2$ of the finger.
The wheel is not actuated, and it can be moved only by interaction with the finger through normal and frictional contact forces.
These forces are modeled via complementarity constraints, which make this a hybrid system.

The goal in the optimal control problem (OCP) is for the unactuated wheel to follow a time-varying reference, with minimal control effort of the finger.
The continuous-time OCP, subject to the finger-wheel complementarity Lagrangian system with friction and impacts, reads as:
\begin{figure}[t]
	\centering
	\includegraphics[width=0.47\textwidth]{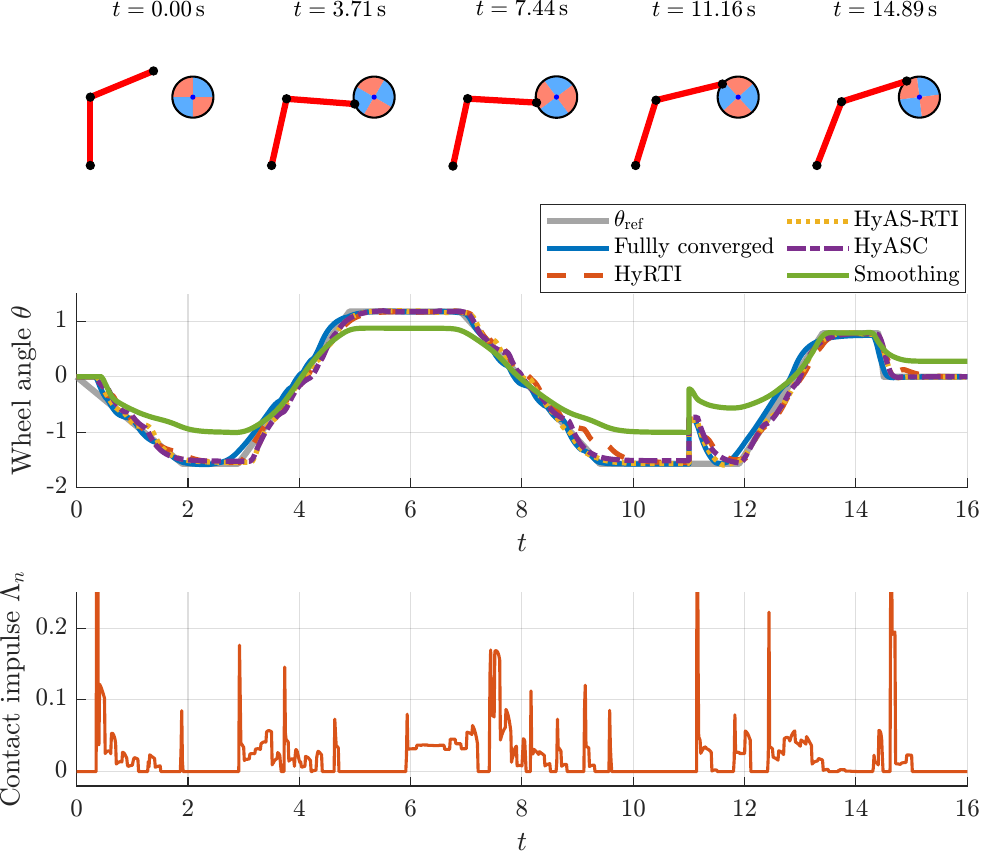}
	\vspace{-0.15cm}
	\caption{
		The top plot shows snapshots of the two-link finger spinning a wheel. A video is available at~\url{https://youtu.be/JJea48uhcyA}. 
		The middle plot compares the wheel reference $\theta_\mathrm{ref}(t)$ with the closed-loop wheel angles for different MPC algorithms. 
		The bottom plot shows the normal contact impulse for HyRTI.
	}
	\vspace{-0.65cm}
	\label{fig:states}
\end{figure}
\begin{mini*}|s|
	{\substack{x(\cdot),u(\cdot)\\\lambda_{\mathrm{n}}(\cdot), \lambda_{\mathrm{t}}(\cdot) }}{
		\displaystyle
		\int_0^T
		\ell(x(t),u(t)) \dd t + E(x(T))
	}{}{}
	\addConstraint{x(0)}{= x_0}{}
	\addConstraint{\dot q(t)}{= v, }{\ t\in [0,T] }
	\addConstraint{\!\!\!\!\!\!\!M(q) \dot v(t)}{\!=\! F(q,u) \!+\! J_\mathrm{n}(q)\lambda_{\mathrm{n}} \!+\! J_\mathrm{t}(q)\lambda_{\mathrm{t}},}{
		\ t\in [0,T] }{}
	\addConstraint{\textnormal{(Velocity jump law)}}{}
	\addConstraint{0\leq }{\lambda_{\mathrm{n}} \perp \varphi(q) \geq 0}{\ t\in [0,T]}
	\addConstraint{\lambda_{\mathrm{t}}}{\in -\mu_{\mathrm{f}} \mathrm{sign}(J_\mathrm{t}(q)^\top v) }{\ t\in [0,T]}
	\addConstraint{u_{\min}}{\le u(t) \le u_{\max},}{\ t\in [0,T].}
\end{mini*}
\noindent Here, $M(q)$ is the joint inertia matrix of the wheel and the robot, which is the main source of nonlinearity.
The vector $F(q,u)$ collects all generalized forces acting on the system, including gravity, Coriolis and control forces.
The function $\varphi(q)$ is the shortest distance between the fingertip and the wheel, and it is complementary to the normal contact force $\lambda_{\mathrm{n}}$.
Further, $J_\mathrm{n}(q) = \nabla_q \varphi(q)$ is the contact normal, and $J_\mathrm{t}(q)$ is the tangent vector at the contact point.
Together they define the contact frame where the normal contact $\lambda_{\mathrm{n}}$, and friction forces $\lambda_{\mathrm{t}}$ act.
The friction force is determined via Coulomb's friction law $\lambda_{\mathrm{t}} \in -\mu_{\mathrm{f}} \mathrm{sign}(J_\mathrm{t}(q)^\top v)$, with friction coefficient $\mu_{\mathrm{f}} = 0.8$.
The set-valued $\mathrm{sign}(\cdot)$ function can be rewritten via linear complementarity constraints~\cite{Brogliato2016}.
Additionally, rigid-body dynamics require a velocity jump law.
We regard here inelastic impacts where the normal contact velocity jumps to zero after each contact, i.e., $J_\mathrm{n}(q)^\top v = 0$.

The initial state is $x_0 = (0,\frac{3\pi}{8},0,0,0,0)$.
The prediction horizon is $T = 1.5$, and the stage cost is defined as $\ell(x(t),u(t)) = (x(t)-x_\mathrm{r}(t))^\top Q (x(t)-x_\mathrm{r}(t)) + u(t)^\top R u(t)$, and the terminal cost as $E(x(T)) = (x(T)-x_\mathrm{r}(T))^\top Q_T (x(T)-x_\mathrm{r}(T))$.
The weight matrices are $Q = \mathrm{diag}(0.5, 0.1, 20,1, 1, 10^{-3}),\ R = \mathrm{diag}(10^{-2},10^{-2}),$ and $Q_T = \mathrm{diag}(1,1,10,0.1,0.1,1)$.
The time-varying reference for $x_\mathrm{r}(t)$ is depicted in middle plot of Fig.~\ref{fig:states}, all other components have a zero reference, other than $x_{\mathrm{r},2} = \frac{3\pi}{8}$.
The torque bounds are $u_{\min} = (-15,-15)$ and $u_{\max} = (15,15)$.

We discretize the OCP using $N=15$ equidistant control intervals, resulting in a control interval and MPC sampling time of $0.1$s.
The optimal control problem is discretized with an implicit Euler time-stepping scheme for Lagrangian systems, cf.~\cite{Brogliato2016,Posa2014}.
The first two control intervals have four integrator steps, the third two, and the remaining ones have one per control interval.
For closed-loop simulation, in contrast, we simulate the dynamics with eight implicit Euler steps, in order to emulate model mismatch via a more accurate simulation.
The control path constraints are enforced pointwise on the control discretization grid.
After discretization, the OCP becomes a parametric MPCC of the form~\eqref{eq:discrete_time_ocp}.
\begin{figure}[t]
	\centering
	\includegraphics[width=0.48\textwidth]{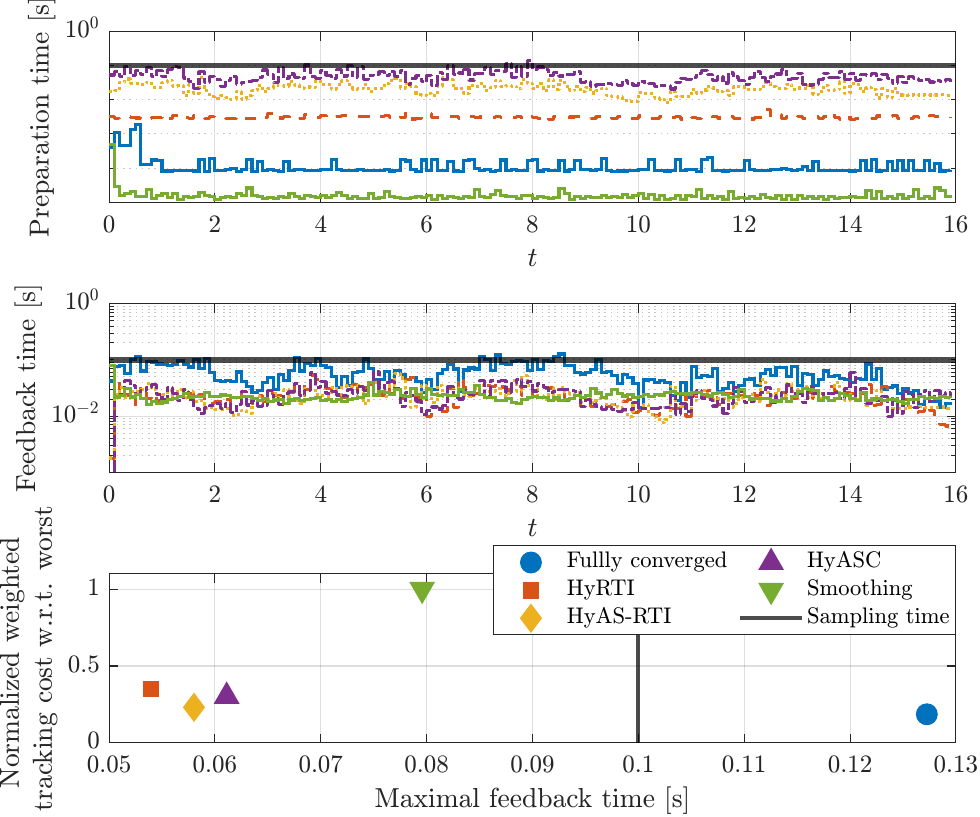}
	\vspace{-0.25cm}
	\caption{
	Preparation and feedback time for various MPC algorithms, and a scaled Pareto plot comparison. 
	}
	\vspace{-0.69cm}
	\label{fig:numerics}
\end{figure}
We perform a closed-loop MPC simulation for $T_s = 16$ seconds, and compare five MPC schemes:
\vspace{-0.20em}
\begin{enumerate}[{A}1.]
	\item Solving the MPCC to full convergence at every iteration with \texttt{CCOpt}~\cite{Pozharskiy2026};
	\item The HyRTI method, with \texttt{CCOpt} as QPCC solver.
	\item The HyASC method, with \texttt{CCOpt} for solving the MPCC to convergence in the preparation phase, and as QPCC solver in the feedback phase.
	\item HyAS-RTI method, with a single SQPCC step in the feedback phase, with \texttt{CCOpt} as QPCC solver.
	\item Smoothed MPCC, which requires solving a standard NLP. 
	We use \texttt{Ipopt}~\cite{Waechter2006}.
	Smoothing is achieved by replacing complementarity constraints $0 \leq G(w) \perp H(w) \geq 0$ by $G(w) \geq 0, H(w) \geq 0$, $G_i(w)H_i(w) = \tau$.
\end{enumerate}
\vspace{-0.1em}
In our experiments, after some tuning we picked $\tau = 10^{-1}$. 
Lower values of $\tau$ yield better approximations but led to frequent failures of the NLP solver.
Since the objective is of least square type, in the QPCCs, we use a Gauss-Newton Hessian approximation.
We tried also commercial MIQP for solving the QPCC, but they were significantly slower than \texttt{CCOpt} and unreliable in convergence, cf. ~\cite[Sec. 6.3.4]{Pozharskiy2026}.

All MPC algorithms are implemented in the MATLAB version of the open-source package \texttt{nosnoc}~\cite{Nurkanovic2024a}, which relies on \texttt{CasADi}~\cite{Andersson2019} for symbolic expressions and automatic differentiation.
The \texttt{nosnoc} repository contains also several other examples from different application domains, which allow for future testing of the proposed methods.\footnote{More examples at \url{https://github.com/nosnoc/nosnoc/tree/main/examples/mpc}.}

In the closed-loop simulation, the MPC controller has to react to model-plant mismatches and to reference changes.
Furthermore, at $t_d = 11$s, we add a disturbance by changing the wheel angle to $\theta(t_d)\!+\!0.5\pi$ to emulate an adversarial external spin.
Fig.~\ref{fig:states} shows the resulting closed-loop wheel position $\theta$ for various MPC controllers.
We see that all proposed MPC algorithms (A2--A4) successfully track the reference and quickly compensate for the disturbance.
Moreover, they have performance comparable to that of the fully converged MPC (A1).
Only the smoothed MPC (A5) has unsatisfactory performance, mostly due to the large approximation error.
The bottom plot of Fig.~\ref{fig:states} shows the resulting normal contact impulses for the HyRTI scheme.
Whenever the contact force jumps from zero to some positive value, a new contact is made, or vice versa, a contact is broken.
This happens quite frequently, and indicates that HyRTI finds numerous switches online that were not prespecified or provided via initialization.

In Fig.~\ref{fig:numerics}, we compare the computation times of our proposed MPC algorithms.
The computations are split into the preparation and feedback phases, cf. Sec.~\ref{sec:mpc}.
The bottom plot of Fig.~\ref{fig:numerics} shows the worst-case feedback time of all algorithms versus the closed-loop tracking cost rescaled w.r.t. the worst performing algorithm, which is smoothing (A5).
Our real-time algorithms (A2-A4) have consistently lower feedback times and are well below the sampling time of 0.1s.
The advanced step algorithms (A3,A4) have more expensive preparation phases due to the additional computations.
Smoothing and the fully converged approach have no significant preparation cost, as the computations cannot be split. 
In conclusion, the proposed hybrid MPC algorithms are real-time feasible, with notable savings in the feedback time compared to a full solve, without notably compromising the tracking performance.
\vspace{-0.25cm}
\section{Conclusion and future work}\label{sec:conclusion}
This developed a comprehensive framework for real-time MPC algorithms for hybrid dynamical systems. 
The key observation is that QP feedback corrections are, in general, insufficient across complementarity active-set changes, whereas QPCC subproblems retain the structure needed to capture switching.
On the theoretical side, we characterized local continuity, and discontinuity properties of parametric MPCC solution maps, showed how one-sided directional derivatives can be obtained from a single QPCC.
We derived local one-step error bounds for SQPCC-based path-following in the presence of solution jumps. 
Based on this, generalizing established MPC paradigms, we proposed several real-time hybrid MPC schemes that use these QPCC corrections in the feedback phase and illustrated their behavior on an example.
An open-source implementation of the algorithm is also provided within \texttt{nosnoc}.

Future work will investigate the stability and robustness properties of our new algorithms.
In classical MPC, Riccati-based structure-exploiting methods led to speed-ups of up to two orders of magnitude over generic solvers~\cite{Frison2020a}.
In our experiments, we used generic solvers. 
Hence, we expect that similar developments for QPCCs and MPCCs will lead to comparable speed ups, and enable new challenging hybrid MPC applications.

\section*{References}
\vspace{-0.6cm}

\raggedbottom
\vspace{-1.2cm}
\begin{IEEEbiography}[{\includegraphics[width=1in,height=1.25in,clip,keepaspectratio]{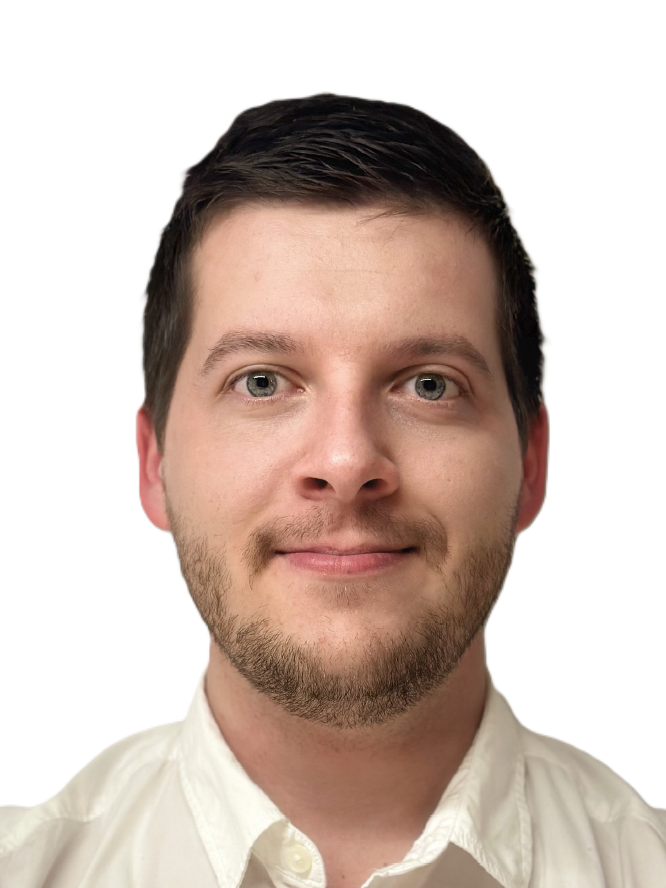}}]{Armin Nurkanovi\'c} received the B.Sc. degree  from the Faculty of Electrical Engineering, Tuzla, Bosnia and Herzegovina, in 2015, and the M.Sc. degree  in Electrical and Computer Engineering, Technical University of Munich, Germany, in 2018. 
In 2023, he received his Ph.D. degree in Engineering from the University of Freiburg, Germany. 
He received the IEEE Control Systems Letters Outstanding Paper Award in 2022 and was a finalist for the 2024 European Systems \& Control PhD Thesis Award. 
In 2025, he served as an interim professor of mathematical optimization at TU Braunschweig. 
His research interests include numerical methods for model predictive control, nonlinear optimization, and control of hybrid systems.
\end{IEEEbiography}
\vspace{-1.1cm}
\begin{IEEEbiography}[{\includegraphics[width=1in,height=1.25in,clip,keepaspectratio]{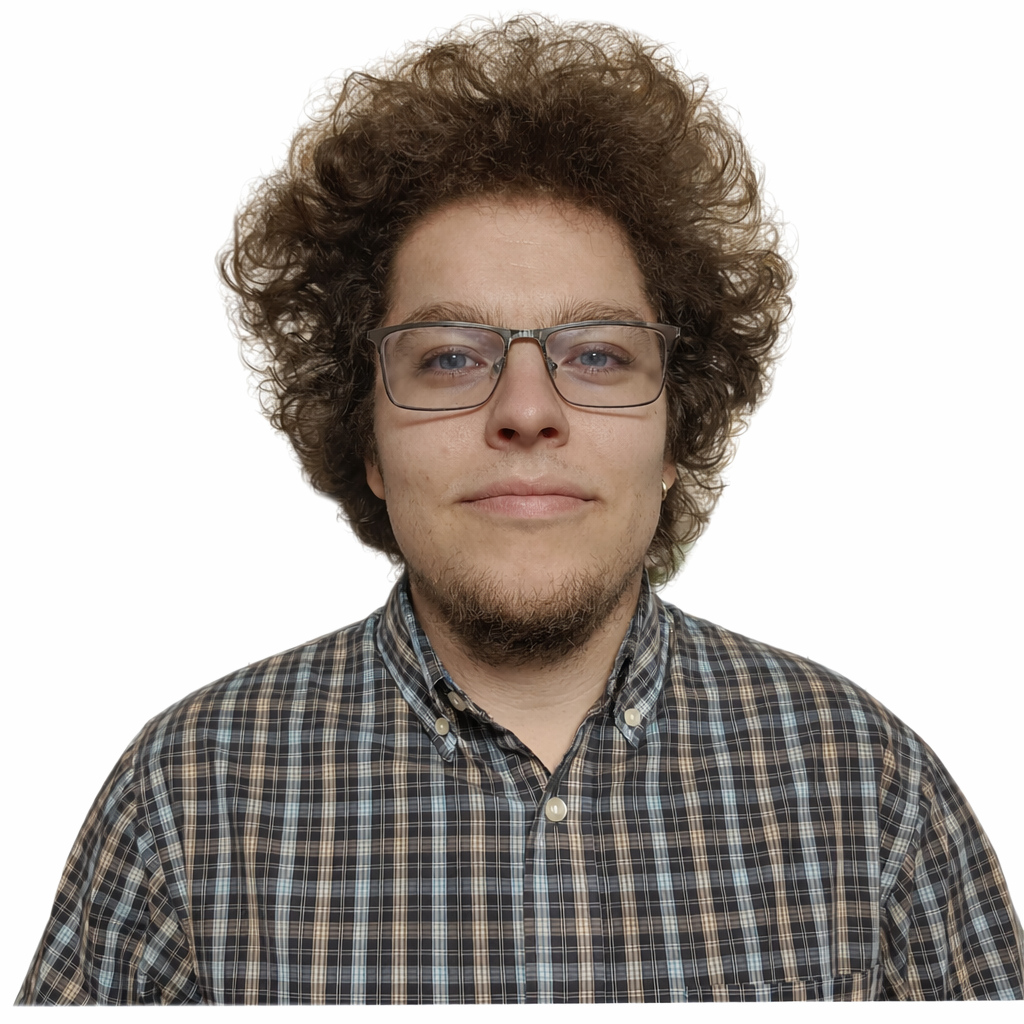}}]{Anton Pozharskiy} is a  Ph.D. student at the University of Freiburg.
	He received a B.Sc degree in Electrical Engineering and a B.Sc degree in Computer Science from the University of Maryland in May 2021. 
	From September 2021 to November 2023, he completed a M.Sc. degree in Embedded Systems Engineering at University of Freiburg.
	His research interests include numerical methods for mathematical programs with complementarity constraints and their application in control.
\end{IEEEbiography}
\vspace{-1.2cm}
\begin{IEEEbiography}[{\includegraphics[width=1in,height=1.25in,clip,keepaspectratio]{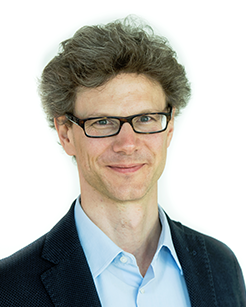}}]{Moritz Diehl} (Member, IEEE) is Professor of Systems Control and Optimization at the University of Freiburg, Germany, where he serves as director of the Department of Microsystems Engineering (IMTEK) and as managing director of the university’s Center for Renewable Energy (ZEE). 
He studied physics and mathematics at Heidelberg University, Germany, and Cambridge University, U.K., in 1993-1999, and received the Ph.D. degree from Heidelberg University in 2001. 
From 2006 to 2013, he was a Professor at the Department of Electrical Engineering, KU Leuven, Belgium. 
Since 2013, he is full professor at the Department of Microsystems Engineering at his current affiliation in Freiburg, where he is also affiliated with the Department of Mathematics. 
His research interests span optimization and control, ranging from numerical method development to applications in various branches of engineering, with a focus on embedded real-time implementations and renewable energy systems. 
\end{IEEEbiography}

\end{document}